\documentclass[reqno,a4paper,11pt]{amsart}
\usepackage{color}
\usepackage{amsmath,amssymb,amsfonts,amsthm,enumerate}
\usepackage{mathrsfs,accents,todonotes}
\usepackage[T1]{fontenc}
\usepackage[active]{srcltx}
\usepackage{epsfig}
\usepackage{graphicx}

\numberwithin{equation}{section}
\newtheorem{theorem}{Theorem}[section]
\newtheorem{lemma}[theorem]{Lemma}
\newtheorem{definition}[theorem]{Definition}
\newtheorem{remark}[theorem]{Remark}
\newtheorem{proposition}[theorem]{Proposition}
\newtheorem{corollary}[theorem]{Corollary}
\newtheorem{example}[theorem]{Example}

\newenvironment{Proof}{\removelastskip\par\medskip 
\noindent{\em Proof.}
\rm}{\penalty-20\null\hfill$\square$\par\medbreak}

\newcommand{\N}{\mathbb{N}}
\newcommand{\R}{\mathbb{R}}

\newcommand{\T}{\mathbb{T}}
\newcommand{\XX}{\mathbb{X}}
\newcommand{\YY}{\mathbb{Y}}

\newcommand{\eps}{\varepsilon}

\newcommand{\Haus}[1]{{\mathscr H}^{#1}} 
\newcommand{\Leb}[1]{{\mathscr L}^{#1}} 
\newcommand{\Borel}[1]{{\mathscr B}(#1)}

\newcommand\uno{{\bf 1}} 

\newcommand{\bb}{{\bf b}}
\newcommand{\BB}{{\bf B}}
\newcommand{\dive}[1]{{\rm div}(#1)}
\newcommand{\Probabilities}[1]{{\mathscr P}\bigl(#1\bigr)}
\newcommand{\Measures}[1]{{\mathscr M}\bigl(#1\bigr)}
\newcommand{\Measuresp}[1]{{\mathscr M}_+\bigl(#1\bigr)}
\newcommand{\ev}{{\sf ev}}
\newcommand{\eeta}{{\mbox{\boldmath$\eta$}}}
\newcommand{\ssigma}{{\mbox{\boldmath$\sigma$}}}

\newcommand{\res}{\mathop{\hbox{\vrule height 7pt width .5pt depth 0pt
\vrule height .5pt width 6pt depth 0pt}}\nolimits}

\newcommand{\trplus}[2]{{\rm Tr}^{{\rm in}}(#1,#2)}

\begin{document}

\title[Regular Lagrangian Flows and well-posedness]{Non-branching properties of Regular Lagrangian Flows\\ and well posedness of the continuity equation\\ in the class of nonnegative solutions}
\author{L. Ambrosio}
\address{Scuola Normale Superiore, Pisa}
\email{luigi.ambrosio@sns.it}
\author{S. Bianchini}
\address{SISSA, Trieste}
\email{bianchin@sissa.it}
\author{F. Renzi}
\address{Scuola Normale Superiore, Pisa}
\email{federico.renzi@sns.it}
\author{F. Vitillaro}
\address{Scuola Normale Superiore, Pisa}
\email{federico.vitillaro@sns.it}
\date{October 2026}

\begin{abstract}
In this paper we deal with the theory of Regular Lagrangian Flows, as initiated in \cite{DL_89} and axiomatized in \cite{Am_2004}. In the first part of the paper we show in Theorem~\ref{thm:main} (see also Corollary~\ref{cor:main}) that the decomposition of $\Leb{1+d}$-almost all the space time in a family of (essentially) disjoint characteristics, the technique at the basis of the well posedness in \cite{ABC14} (for planar, divergence-free and autonomous vector fields), and in \cite{Bia_Bo_2020} (for vector fields having a suitable regularity in space, including BV vector fields) is not only sufficient for well posedness, but also necessary. This shows that in many cases the
typical uniqueness statement of the DiPerna-Lions theory, namely that
the path $\XX(\cdot,x)$ is (essentially) unique given the initial position $x$, can, as in the classical theory, be
strengthened by saying that the paths $\XX(\cdot,x)$ (essentially) do not branch. 
As the simple Example~\ref{example:nogood} shows, when we consider vector fields with singular divergence well posedness results may hold only for solutions that do not change sign. In the second part of the paper we revisit \cite{Am_Cri_Man_2005,Am_Del_Maly_2007} obtaining a well posedness
result for special vector fields of bounded deformation and nonnegative solutions, under a mild
regularity assumption \eqref{eq:time_regularity} with respect to the time variable, weaker
than the space-time $SBD$ regularity assumption made in \cite{Am_Del_Maly_2007} for nearly incompressible vector fields.
Finally, we revisit carefully the constructions of
Section~9 of \cite{Bia_Bo_2020}, to see how they provide existence and uniqueness for \eqref{eq:CE} assuming $BV$ regularity in space, in the class of nonnegative bounded solutions.
\end{abstract}

\maketitle

\tableofcontents

\section{Introduction}

Throughout this paper, $T\in (0,\infty)$ and a Borel vector field $\bb:(0,T)\times\R^d\to\R^d$ satisfying
\begin{equation}\label{eq:growth}
\frac{|\bb|}{1+|x|}\in L^1\bigl((0,T);L^\infty(\R^d)\bigr)+L^1\bigl((0,T);L^1(\R^d)\bigr)
\end{equation}
are given, so that we can write $\bb=\bb^1+\bb^2$, where
$$\frac{|\bb^1|}{1+|x|}\in L^1\bigl((0,T);L^1(\R^d)\bigr),\quad \frac{|\bb^2|}{1+|x|}\in L^1\bigl((0,T);L^\infty(\R^d)\bigr).$$
Occasionally we use the notation $\bb_t(x)=\bb(t,x)$.
We consider the associated continuity equation
\begin{equation}\label{eq:CE}
\frac{d}{dt}\rho_t+\dive{\bb_t\rho_t}=0\qquad\text{in $(0,T)\times\R^d$}
\end{equation}
always understood in the sense of distributions (see Section~\ref{sec:prem} for more specifications, including also the Cauchy initial value problem)
and the corresponding ODE
\begin{equation}\label{eq:ODE}
\frac{d}{dt}\gamma(t)=\bb(t,\gamma(t))=\bb_t(\gamma(t)),
\end{equation}
where by solution we always mean an integral absolutely continuous map $\gamma:[0,T]\to\R^d$.

As we will be mostly concerned with the uniqueness issue, many other variants are possible and will not be considered here, for instance replacing $\R^d$ with the 
$d$-dimensional torus or replacing $\R^d$ by sufficiently regular domains $\Omega\subseteq\R^d$ such that $\bb_t\cdot\nu_{\Omega}=0$ on $\partial\Omega$ 
in the sense of distributions (in the latter case, the results can be obtained from the global case just by extending $\bb_t$ to $0$ out of $\Omega$).

\medskip

After the pioneering work \cite{DL_89} dealing with Sobolev vector fields, the theory of Regular Lagrangian Flows, introduced in \cite{Am_2004} provides
a general framework to establish uniqueness of bounded solutions to \eqref{eq:CE}, consistency with the Lagrangian representation
provided by solutions to \eqref{eq:ODE} as well as stability results. Its relevance includes also more general PDE's, as for instance
the Navier-Stokes equation, see for instance the recent paper \cite{Galeati}.  The principle underlying the RLF axiomatization is the possibility to select
characteristics (i.e. solutions to the ODE) in such a way that for all times $t\in [0,T]$ there is (in a quantitative way) no concentration. Under suitable
regularity assumptions on the vector fields (which include for instance $BV$ regularity \cite{Am_2004,Bia_Bo_2020}, split structure \cite{MR3569235, LeBris_Lions_2004, Champagnat_Jabin_2010},  convolution with singular integrals \cite{MR3078074}) it is known that this no concentration constraint leads to $\Leb{d}$-pathwise uniqueness, namely there is a 
$\Leb{d}$-essentially unique selection of characteristics $\XX(t,x)$, the RLF, satisfying this property. 

\medskip

The aim of this paper is threefold: first we review different axiomatizations of RLF, with the necessity to include as in \cite{Am_2004} also
{\it generalized} RLF to obtain existence and uniqueness of RLF (as in Optimal Transport, where existence and uniqueness of transport maps can come as a consequence
of uniqueness of optimal plans), and show how nonbranching properties stronger than pathwise uniqueness come just as a consequence; then, revisiting
the Eulerian techniques of \cite{Am_Cri_Man_2005,Am_Del_Maly_2007}, we provide a new well posedness result for space-time vector fields $\bb$ of bounded deformation ($BD$) having $SBD$ regularity with respect to
the space variable and possibly singular divergence (recall that even for the case of divergence-free $BD$ vector fields the existence of the RLF is still
an open problem); finally, we revisit the fundamental techniques introduced in \cite{Bia_Bo_2020} (Lagrangian representations, reduction to a family of 1-dimensional
problems via partition by characteristics)  
to see how they provide uniqueness for \eqref{eq:CE} in the class of nonnegative bounded solutions. We exploit
also the more recent superposition principle \cite{ARV} for 1-currents with locally finite mass to cover vector fields that are not compactly supported and we relax
the bounded compression condition of \cite{Am_Del_Maly_2007} and \cite{Bia_Bo_2020} to a one-sided one, see Definition~\ref{def1sidednearlyincompressible}, still sufficient for forward uniqueness in the class of nonnegative solutions.

\medskip

In connection with the first issue, notice that in the classical theory of ODE's backward uniqueness for {\it all} final conditions grants not only pathwise uniqueness, but also the
property that solutions $\XX(t,x),\,\XX(t,y)$ starting at $t=0$ from $x\neq y$ do not intersect in the future. So we may ask ourselves whether
the GRLF axiomatization is sufficiently strong to provide also some control on the rate of intersection of the characteristics $\XX(t,x)$.
We provide a positive answer in Corollary~\ref{cor:main}, in the simpler case when we have forward and backward uniqueness for the PDE (see
Definition~\ref{deffnb} and Theorem~\ref{thm:main} for the more precise one-sided version): indeed, we
prove that there is a $\Leb{d}$-negligible set $\mathcal N$ of initial conditions such that $\{\XX(t,x)\}_{x\in\R^d\setminus\mathcal N}$
are pairwise disjoint. 
 In particular, if the transported density $\rho$ is everywhere positive, these characteristics provide a space-time partition of $\Leb{1+d}$-almost all of $(0,T)\times\R^d$. 
The idea of the proof is inspired by \cite{Bia_Bo_2020}, where under some regularity assumptions on the vector field satisfied by $BV$
vector fields (namely the existence of sufficiently many so-called ``cylinders of approximate flow'') an untanglement
procedure is built to provide a partition in characteristics, suitable to prove uniqueness of \eqref{eq:CE} by disintegration (and 
then uniqueness of the GRLF comes as a consequence). In the subsequent paper \cite{BB21} it is shown how just a forward untanglement
property (see Definition~\ref{deffnb}) is sufficient to provide uniqueness (this allows to cover other classes of vector fields). 
Here we prove, by reversing the procedure, that this untanglement
property is also necessary. Namely assuming that for some reason we know
the GRLF to be unique, the untanglement just comes as an additional property; in this connection, recall that for some
classes of vector fields the uniqueness of the GRLF is known to come only from Eulerian techniques, via well posedness of \eqref{eq:CE}. Hence,
this abstract result provides one more bridge between the Eulerian and Lagrangian viewpoints. It is also
crucial, for the validity of the proof, to assume uniqueness at the level of the GRLF, as the measures we build on pairs of intersecting
characteristics (supplied by duality in Optimal Transport, as in \cite{Bia_Bo_2020}) are not necessarily concentrated on a graph.

\medskip

In connection with the second and third issues, we notice that the results of \cite{DL_89}, \cite{Am_2004} and of many subsequent papers grant uniqueness of solutions
to \eqref{eq:CE} even in the class of bounded signed solutions $\rho$, via the so-called renormalization property. However, as detailed in Theorem~\ref{thm:main}(a)(b),
 uniqueness of the GRLF is equivalent to uniqueness of solutions to \eqref{eq:CE} in the class of nonnegative $\rho$, so it is natural to look for sharper uniqueness results within this class. This problem is also closely related to the bounds on the distributional divergence $D\cdot\bb_t$ of $\bb_t$: the $L^\infty$ bound on divergence in \cite{DL_89} was already relaxed in \cite{Am_2004}, assuming $D\cdot\bb_t\ll\Leb{d}$ and with negative part of the density in
$L^\infty$. Since the condition $D\cdot\bb_t\ll\Leb{d}$ is not quantitative, it is natural to investigate what happens when
we allow $D\cdot\bb_t$ to be singular, but bounded from below by a constant multiple of $\Leb{d}$. This is precisely the case when
uniqueness of \eqref{eq:CE} may fail in the class of signed solutions, as the following simple example shows:

\begin{example} \label{example:nogood} Let $d=1$, $a\in\R$ and set
$$
\rho_t(x):=
\begin{cases}
 \,\,a          &\text{if $0<x<t$,}\\
 -a          &\text{if $-t<x<0$,}\\
 \,\,0           &\text{if $|x|\geq t$.}
\end{cases}
$$
Then $\rho_t$ solves \eqref{eq:CE} with $\rho_0=0$ and $\bb(x)={\rm sign}(x)$. 
\end{example}

Obviously this ill-posedness example can be replicated in higher dimensions. Via the concept of near incompressibility
(see Definition~\ref{def:nivf}) the bounds on the divergence of $\bb_t$ have already been relaxed, putting them in integral form, providing well
posedness results in \cite{Am_Cri_Man_2005,Am_Del_Maly_2007} and in \cite{Bia_Bo_2020}, respectively for $SBD$ vector fields and $BV$ vector fields. As we wrote,
we refine the results \cite{Am_Cri_Man_2005,Am_Del_Maly_2007} via Eulerian techniques and we revisit carefully the constructions of
Section~9 of \cite{Bia_Bo_2020}, to see how they provide existence and uniqueness for \eqref{eq:CE} in the class of nonnegative bounded solutions.  We observe that the vector field in the (prototypical) Example \ref{example:nogood} is easily seen not to be nearly incompressible, but still falls within our hypotheses.

\bigskip
{\bf Acknowledgements.} L.A. gratefully acknowledges the support of ITS of ETH, where part of this paper has been written during his stay as visiting scientist. L.A., F.R. and F.V. are partially supported by a grant of the Balzan
Foundation. F.R. and F.V. are members of INdAM-GNAMPA.

\bigskip{\bf  Use of AI.}  No AI tools were used in producing the mathematical content or text of this manuscript. AI tools were used only for
language editing and a final review of the paper.

\section{Preliminary definitions and results}\label{sec:prem}

{\bf Main notation.} The characteristic function of a set $E$ will be denoted by $\uno_E$. 

The space of Borel probability measures in $X$
and of nonnegative $\sigma$-additive Borel measures will be denoted respectively by $\Probabilities{X}$ and $\Measuresp{X}$,
while $\Measures{X}$ denotes the class of signed $\sigma$-additive measures on bounded Borel sets with locally finite total variation. The restriction
operator $\mu\mapsto\mu\res E$ induced by a Borel set $E\subseteq X$ and mapping $\Measuresp{X}$ to $\Measuresp{X}$ is defined by
$\mu\res E:=\uno_E\mu$.
The push forward operator $f_\#:\Measuresp{X}\to\Measuresp{Y}$ induced by a Borel map $f:X\to Y$ is defined by $f_\#\mu(E):=\mu(f^{-1}(E))$
for any Borel set $E\subseteq Y$. Notice that the operator $f_\#$ maps probability measures to probability measures.

If $\Omega\subseteq\R^d$ is an open set, the approximate jump set of a function $f\in L^1_{loc}(\Omega,\R^m)$ is the set of points $x\in\Omega$
such that the rescaled functions $f_{x,r}(y)=f(x+ry)$ converge in $L^1_{\rm loc}(\R^d)$ to the function
$$
\phi_{a,b,\nu}(y):=
\begin{cases}
a &\text{if $\langle y,\nu\rangle>0$};\\
b &\text{if $\langle y,\nu\rangle<0$,}
\end{cases}
$$
for some $a,\,b\in\R^m$ and $\nu\in {\bf S}^{d-1}$ with $a\neq b$. Remarkably, this set is Borel and countably $\Haus{d-1}$-rectifiable, see \cite{DelNin}.

We shall denote by $\ev_t:C([0,T];\R^d)\to\R^d$ the canonical evaluation map at time $t\in [0,T]$, namely $\ev_t(\gamma)=\gamma(t)$.

\medskip

\subsection{$BV$ and $BD$ vector fields}

In this section we introduce the basic notation and concepts relative to the $BV$ space of functions of bounded variation
and the $BD$ space of vector fields of bounded deformation in an open set $\Omega\subseteq\R^k$.
For the former a reference book is \cite{AFP_00}, for the latter the paper \cite{ACD_97}.

\begin{definition} [BV functions]
We say that $f\in L^1(\Omega)$ belongs to $BV(\Omega)$ if there exists a vector valued measure
$Df=(D_1f,\ldots,D_kf)$ with finite total variation representing the derivative of $f$ in the sense of
distributions, namely
$$
\int_\Omega f\frac{\partial\phi}{\partial x_i}\,dx=-\int_\Omega\phi\,dD_i f\qquad
\forall\phi\in C^\infty_c(\Omega),\,\,\,i=1,\ldots,k.
$$
\end{definition}

The definition can be extended, arguing componentwise, to vector fields. We shall denote by $D^af+D^s f$ the decomposition of 
$Df$ in absolutely continuous and singular part with respect to $\Leb{k}$ and say that $f\in SBV(\Omega)$ if $D^s f$
is concentrated on a countably $\Haus{k-1}$-rectifiable set.

\begin{definition}[BD functions]
We say that $\bb\in L^1(\Omega;\R^k)$ belongs to $BD(\Omega)$ if there exists a symmetric matrix of measures
$E\bb= (E\bb)_{ij}$, $1\leq i,j\leq k$, representing the symmetric part of the distributional derivative of $\bb$, namely
$$
\frac{1}{2}\int_\Omega \bigl(\bb_i\frac{\partial\phi}{\partial x_j}+\bb_j\frac{\partial\phi}{\partial x_i}\bigr)\,dx=
-\int_\Omega\phi\,dE_{ij}\bb\qquad\forall\phi\in C^\infty_c(\Omega),\,\,\,i,j=1,\ldots,k.
$$
\end{definition}

Also in this case we shall denote by $E^a \bb+E^s \bb$ the decomposition of 
$E\bb$ in absolutely continuous and singular part with respect to $\Leb{k}$ and say that $\bb\in SBD(\Omega)$ if $E^s \bb$
is concentrated on a countably $\Haus{k-1}$-rectifiable set.

\smallskip

In the following proposition we recall the structural properties of $BV$ and $BD$ spaces needed in this paper.

For any unit vector $\xi\in\R^k$ we denote by $\pi_\xi$ the hyperplane orthogonal to $\xi$ and
$$
\Omega_{y,\xi}=\{s\in\R: y+s\xi\in\Omega\},
$$ 
$$
\bb_{y,\xi}(s):=\bb(y+s\xi)\cdot\xi\qquad y\in\pi_\xi,\,\,\,s\in\Omega_{y,\xi}\subseteq\R.
$$

The following structural theorem provides detailed information on the structure of $E\bb$ in terms
of the one-dimensional distributional derivatives of $\bb_{y,\xi}$. We focus mostly on the jump
part, the relevant one for our purposes.

\begin{proposition}\label{structural}
For any unit vector $\xi$ one has $\langle E\bb\xi,\xi\rangle=\int_{\pi_\xi} D\bb_{y,\xi}\,d\Haus{k-1}(y)$, namely
$$
\int\phi(x)\, d\langle E\bb\xi,\xi\rangle(x)=\int_{\pi_\xi}\int_{\Omega_{y,\xi}}\phi(y+s\xi)\,dD\bb_{y,\xi}(s)\,d\Haus{k-1}(y)
$$
for any bounded Borel function $\phi$ with compact support. Analogously, 
$$
\langle E^a\bb\xi,\xi\rangle=\int_{\pi_\xi} D^a\bb_{y,\xi}\,d\Haus{k-1}(y), 
\qquad
|\langle E^a\bb\xi,\xi\rangle|=\int_{\pi_\xi} |D^a\bb_{y,\xi}|\,d\Haus{k-1}(y),
$$
$$
\langle E^s\bb\xi,\xi\rangle=\int_{\pi_\xi} D^s\bb_{y,\xi}\,d\Haus{k-1}(y),
\qquad
|\langle E^s\bb\xi,\xi\rangle|=\int_{\pi_\xi} |D^s\bb_{y,\xi}|\,d\Haus{k-1}(y).
$$
Finally, 
\begin{equation}\label{eq:sliceJ1}
\int_{\pi_\xi}\Haus{0}\res J_{\bb_{y,\xi}}\,d\Haus{k-1}(y)\leq\Haus{k-1}\res J_{\bb}
\end{equation}
and $\bb\in SBD(\Omega)$ implies $(\bb)_{y,\xi}\in SBV(\Omega_{y,\xi})$ for $\Haus{k-1}$-a.e. $y\in\pi_\xi$.
\end{proposition}

Being $(J_\bb)_{y,\xi}=\{s\in\Omega_{y,\xi}:\ y+s\xi\in J_\bb\}$, from \eqref{eq:sliceJ1}, it follows that 
\begin{equation}\label{eq:sliceJ2} 
J_{\bb_{y,\xi}}\subseteq (J_\bb)_{y,\xi} \qquad\text{for $\Haus{k-1}$-a.e. $y\in\pi_\xi$.}
\end{equation}

\medskip

\subsection{Divergence-measure fields}

In this section we assume that $C:\Omega\subseteq\R^k\to\R^k$ is a locally bounded vector field
whose divergence, in the sense of distributions, is a locally finite Radon measure in
$\Omega$, denoted by $D\cdot C$. We will also denote by $D^a\cdot C$ and $D^s\cdot C$, respectively the absolutely
and singular parts of $D\cdot C$ with respect to $\Leb{k}$.

We denote by ${\mathcal M}_\infty(\Omega)$ the class of these
vector fields. Whenever we consider an open domain with $C^1$ boundary, $\nu_\Omega$ will stand
for its inner unit normal.

Given a domain with $C^1$ boundary $\Omega'\Subset\Omega$ we can 
define the ``inner'' trace of the normal component of $C$ on $\partial\Omega'$ as a distribution 
as follows:
\begin{equation}\label{trace}
-\langle\trplus{C}{\partial\Omega'},\varphi\rangle:=
\int_{\Omega'}\nabla\varphi\cdot C\,dx+\int_{\Omega'}\varphi\,d D\cdot C
\qquad\text{for any $\varphi\in C^\infty_c(\Omega)$.}
\end{equation}
This definition is obviously consistent with the Gauss--Green formula, in the case when
$C\in C^1(\overline{\Omega'},\R^k)$, and gives that the distribution is induced by the integration 
on $\partial\Omega'$ of $C\cdot\nu_{\Omega'}$.

It turns out (see \cite{Am_Cri_Man_2005}, \cite{ChenFrid} for a more detailed discussion) that this distribution is induced by the integration of an $L^\infty$ 
function on $\partial\Omega'$, that we will still denote by $\trplus{C}{\partial\Omega'}$,
and moreover this function depends only on $\partial\Omega'$ and its orientation, 
rather than on $\Omega'$. Namely if $\Sigma$ is a Borel set contained in $\partial\Omega_1\cap\partial\Omega_2$
and if $\nu_{\Omega_1}=\nu_{\Omega_2}$ on $\Sigma$, then
\begin{equation}\label{local}
\trplus{C}{\partial\Omega_1}=\trplus{C}{\partial\Omega_2}\qquad\text{$\Haus{k-1}$-a.e. on $\Sigma$.}
\end{equation}
Property \eqref{local}, together with the absolute continuity property
\begin{equation}
\text{$|D\cdot C|(E)=0$ for any $\Haus{k-1}$-negligible set $E\subseteq\Omega$}
\end{equation}
(see for instance \cite[Proposition~3.4]{Am_Cri_Man_2005}) can be used to define the trace of $C$
on any oriented $\Haus{k-1}$-rectifiable set $\Sigma\subset\Omega$ (with the orientation
provided by a unit normal to the approximate tangent space), as $\Haus{k-1}$-almost all of $\Sigma$ 
can be covered by a sequence $(K_i)$ of compact sets of $C^1$ oriented hypersurfaces. This
construction is independent, up to $\Haus{k-1}$-negligible sets, of the choice of $K_i$, but it does depend
on the orientation, and we shall adopt the notation $\trplus{C}{\nu_\Sigma}:\Sigma\to\R$ to emphasize this dependence. 

In particular, one has the following result (see for instance \cite[Proposition~3.4]{Am_Cri_Man_2005}):

\begin{proposition}\label{pbasic}
If $C\in {\mathcal M}_\infty(\Omega)$, then
$$ (D\cdot C)\res\Sigma=(\trplus{C}{\nu_\Sigma}+\trplus{C}{-\nu_\Sigma})\Haus{k-1}\res\Sigma$$
for any countably $\Haus{k-1}$-rectifiable set $\Sigma\subseteq\Omega$.\\ In particular,
$|D\cdot C|(\Sigma)=0$ implies $\trplus{C}{\nu_\Sigma}=-\trplus{C}{-\nu_\Sigma}$ 
$\Haus{k-1}$-a.e. on $\Sigma$.
\end{proposition}

Finally, we recall the chain rule for traces (see \cite[Theorem~4.2]{Am_Cri_Man_2005}):

\begin{theorem} [Chain rule for traces]\label{change_traces}
If $w\in L^\infty_{\rm loc}(\Omega)$,  $C\in L^\infty_{\rm loc}\cap BD_{\rm loc}(\Omega)$ and 
$wC$ belongs to ${\mathcal M}_\infty(\Omega)$, then for any function $h\in C^1(\R)$ and any
oriented $\Haus{k-1}$-rectifiable set $\Sigma\subseteq\Omega$ one has
$$
\trplus{h(w)C}{\nu_\Sigma}=h\left(\frac{\trplus{wC}{\nu_\Sigma}}
{\trplus{C}{\nu_\Sigma}}\right)\trplus{C}{\nu_\Sigma}
\qquad\text{$\Haus{k-1}$-a.e. on $\Sigma$,}
$$
where the ratio $\trplus{wC}{\nu_\Sigma}/\trplus{C}{\nu_\Sigma}$ is arbitrarily defined at points 
where the trace $\trplus{C}{\nu_\Sigma}$ vanishes.
\end{theorem}

\subsection{Renormalization}
 
In this section we assume that $\Omega=I\times A$ with $I\subseteq\R$ open interval and $A\subseteq\R^d$ open, and that
$\BB:\Omega\to\R^{d+1}$ has the form 
\begin{equation}\label{eq:newform}
\BB=(1,\bb)\quad\text{with}\quad \bb\in L^1_{\rm loc}\bigl(I;BD_{\rm loc}(A)\bigr)\cap L^\infty_{\rm loc}(\Omega;\R^d),
\end{equation}
so that, in particular, $\BB\in {\mathcal M}_\infty(\Omega)$ and the results of the previous section are applicable to $\BB$ (with $k=d+1$).
Arguing as in \cite{Am_Cri_Man_2005}, we analyze the behaviour of the commutators
\begin{equation}\label{teps}
T_\eps:=\left(D\cdot (\BB w)\right)\ast\rho_\eps- D\cdot (\BB (w\ast\rho_\eps)),
\end{equation}
where $w\in L^\infty_{\rm loc}(\Omega)$, $\rho:\R^d\to\R$ is an even, compactly supported convolution kernel 
and $\rho_\eps(z)=\eps^{-d}\rho(z/\eps)$, so that the convolution acts only with respect to the space variable. 
We define
\begin{equation}\label{Irho}
I(\rho):=\int_{\R^d}|z||\nabla\rho(z)|\,dz.
\end{equation}

The following result, stated in \cite[Theorem~2.6]{Am_Cri_Man_2005} under the assumption that $\BB\in BD_{\rm loc}(\Omega)$, is still
true under the weaker assumption \eqref{eq:newform}, using the special structure of $\BB$ and our choice of the commutators.

\begin{theorem} [Concentration of commutators]\label{tconcentra}
Let $\BB$ as in \eqref{eq:newform} and let $w\in L^\infty_{\rm loc}(\Omega)$ such that $w\BB\in\mathcal{M}_{\infty}(\Omega)$.
Let $\rho_\eps$ be a family of mollifiers in space induced by a radial, even convolution
kernel $\rho$. Then:
\begin{itemize}
\item[(i)] The distributions $T_\eps$ defined by \eqref{teps}
are induced by measures with locally uniformly bounded variation in
$\Omega$ as $\eps\downarrow 0$.
\item[(ii)] Any limit point, in the distribution sense, of $|T_\eps|$ as $\eps\downarrow 0$
is a measure $\sigma$ with locally finite variation in $\Omega$ satisfying

$$
\sigma(A)\leq\Vert w\Vert_{L^\infty(A)}\bigl[ I(\rho) \Leb{1} \otimes |E^s \bb_t|\,(A)+|D^s\cdot \BB|(A)\bigr]
\quad\text{for any open set $A\Subset\Omega$.}
$$
\end{itemize}
\end{theorem}

Using Theorem~\ref{tconcentra}, multiplying both sides by $h'(w\ast\rho_\epsilon)$ and passing to the limit as $\epsilon\to 0$
one obtains that $h(w)\BB\in {\mathcal M}_\infty(\Omega)$ for any $h\in C^1(\R)$, more precisely
for any open set $\Omega'\Subset\Omega$, with $M=\Vert w\Vert_{L^\infty(\Omega')}$, one has
\begin{eqnarray}\label{eq:verygoodac}
|D\cdot (h(w)\BB)|&\leq&\hskip -10pt\sup_{z\in [-M,M]}|h'(z)| \left(MI(\rho)\int |E^s \bb_t| \, dt+M|D^s\cdot \BB|+|D\cdot (w\BB)|\right)
\\&& + \sup_{z\in [-M,M]}|zh'(z)-h(z)||D\cdot\BB|\qquad\text{in $\Omega'$.}\nonumber
\end{eqnarray}
In addition, as the ``defect'' measure $\sigma$ in (ii) is singular with respect to $\Leb{d+1}$, one obtains
\begin{equation}\label{goodac}
D^a\cdot (h(w)\BB)=(h(w)-wh'(w))D^a\cdot \BB+h'(w)D^a\cdot (w\BB).
\end{equation}

\subsection{Weak solutions to \eqref{eq:CE} and the Cauchy problem} 

Let $\rho\in L^1_{\rm loc}((0,T)\times\R^d) \cap L^\infty_{\rm loc}((0,T)\times\R^d)$ and $\rho_t=\rho(t,\cdot)$. 
Without loss of generality, from now on we shall always consider weakly continuous solutions $(0,T)\ni t\mapsto\rho_t$ of 
\eqref{eq:CE} (in duality with $C^1_c(\R^d)$), 
this is granted by the following proposition (see for instance \cite{AmCri_2014}).

\begin{proposition} \label{prop:trivial} Let $\bb$ be satisfying \eqref{eq:growth}. Then the following conditions are
equivalent:
\begin{itemize}
\item[(a)] $\rho_t$ solves \eqref{eq:CE} in the sense of distributions in $(0,T)\times\R^d$;
\item[(b)] there exists a weakly continuous representative $\tilde{\rho}_t$ of $\rho_t$, $t\in [0,T]$, satisfying
\begin{equation}\label{eq:weak1}
\int_{\R^d}\phi\tilde\rho_t\,dx-\int_{\R^d}\phi\tilde\rho_s\,dx=\int_s^t\int_{\R^d}\bb_\tau\cdot\nabla\phi\tilde\rho_\tau\,dx\, d\tau
\quad\text{for all $\phi\in C^1_c(\R^d)$, $0\leq s\leq t\leq T$.}
\end{equation}
\end{itemize}
\end{proposition}

\begin{definition}[Cauchy problem for \eqref{eq:CE}]\label{def:Cauchy}
Given $\bar\rho\in L^1\cap L^\infty(\R^d)$, we say that $\rho_t$ solves the initial value problem 
\begin{equation}\label{eq:CE_initial}
\begin{cases}
\displaystyle{\frac{d}{dt}\rho_t+\dive{\bb_t\rho_t}}=0&\text{in $(0,T)\times\R^d$;}\cr\cr
\rho_0=\bar\rho
\end{cases}
\end{equation}
if \eqref{eq:weak1} holds and $\lim\limits_{t\to 0}\int_{\R^d}\phi\rho_t\,dx=\int_{\R^d}\phi\bar\rho\,dx$ for all $\phi\in C^1_c(\R^d)$.
\end{definition}

\begin{remark}[Extension to negative times and space-time vector fields $\rho\BB=(\rho,\rho\bb)$]\label{rem_negative_times}
{\rm A standard procedure, that we will also adopt in the final sections of the paper, consists in the extension of $\bb$ to negative
times, setting $\bb_t=0$ and $\rho_t=0$ for all $t\in (-\infty,0)$. In this case, Definition~\ref{def:Cauchy} can be formulated by saying that 
for all $\phi\in C^1_c(\R^d)$ the map $t\mapsto\int_{\R^d}\phi\rho_t\,dx$
is a $SBV_{\rm loc}$ function in $(-\infty,T)$, with distributional derivative
$$
\biggl(\int_{\R^d}\phi\bar\rho\,dx\biggr)\,\delta_0+\biggl(\int_{\R^d}\bb_t\cdot\nabla\phi\rho_t\,dx\biggr)\Leb{1}.
$$

It will also be convenient to formulate \eqref{eq:CE} when a nonzero right hand side is involved, possibly a measure $\theta$ in
$(0,T)\times\R^d$ with finite total variation on bounded sets. 
In this case, extending both $\bb$ and $\rho$ to $0$ for $t\in (0,T)^c\times\R^d$, an equivalent formulation
involves the space-time vector field
\begin{equation}\label{eq:spacetime}
\BB(t,x)=\bigl(1,\bb_t(x)\bigr)
\end{equation}
as follows: 
\begin{equation}
\begin{cases}
\displaystyle{\frac{d}{dt}\rho_t+\dive{\bb_t\rho_t}}=\theta&\text{in $(0,T)\times\R^d$;}\cr\cr
\rho_0=\bar\rho
\end{cases}
\end{equation}
holds if and only if
\begin{equation}\label{eq:rep_div_BB}
D\cdot (\rho\BB)=\bar\rho\Haus{d}\res\{t=0\}+\theta-\rho_T\Haus{d}\res\{t=T\}
\qquad\text{in $\R\times\R^d$,}
\end{equation}
where from now on $D\cdot {\bf C}$ stands for the space-time divergence of a $(d+1)$-dimensional vector field ${\bf C}$.
}\end{remark}

Let us define
\begin{equation}\label{classR}
{\mathcal R}:=\biggl\{\rho_t\in L^\infty\bigl((0,T);L^1(\R^d)\bigr)\cap L^\infty\bigl((0,T); L^\infty(\R^d)\bigr):
\,\,\text{$\rho_t$ is weakly continuous}\biggr\},
\end{equation}
where weak continuity is meant, in agreement with Proposition~\ref{prop:trivial}(b), in duality with $C^1_c(\R^d)$.
Notice however that the uniform $L^\infty$ bound in this class allows to upgrade immediately this weak continuity to
continuity in duality with $L^1(\R^d)$. We shall also denote by ${\mathcal R}_+$ the subset of ${\mathcal R}$ consisting of nonnegative functions. 

\smallskip

Under assumption \eqref{eq:growth}, the weak continuity  requirement in \eqref{classR} is essentially not restrictive,
thanks to Proposition~\ref{prop:trivial}. In addition, for the class of solutions we are mostly interested in, namely $\rho\in {\mathcal R}_+$, the total mass starting from 
an initial $\bar\rho\in L^1(\R^d)$ is conserved, under assumption \eqref{eq:growth}:
$$
\int_{\R^d}\rho_t\,dx=\int_{\R^d}\bar\rho\,dx \qquad\text{for all $t\in [0,T]$.}
$$
Indeed, it is sufficient to build 
functions $\phi_R\in C^\infty_c(\R^d)$ with $0\leq\phi_R\leq 1$, $\phi_R\mid_{B_R(0)}\equiv1$ and
$(1+|x|)|\nabla\phi_R(x)|$ uniformly bounded with respect to $x$ and $R$, noticing that these conditions grant
\begin{eqnarray*}
\biggr|\int_{\R^d}\phi_R(\rho_s-\bar\rho)\,dx\biggr|&\leq&
\int_0^s\biggl|\int_{\R^d}\bb_t\cdot\nabla\phi_R\rho_t\,dx\biggr|\,dt
\leq C\|\rho\|_\infty\int_0^T \int_{\{\phi_R<1\}} \frac{|\bb^1_t|}{1+|x|}\,dx\,dt
\\&+& C\int_0^T \|\frac{|\bb^2_t|}{1+|x|}\|_\infty\int_{\{\phi_R<1\}} \rho_t\,dx\,dt
\quad\forall s\in [0,T]
\end{eqnarray*}
which, since the right hand side tends to $0$ as $R\to\infty$ by dominated convergence, implies the conservation of mass. Hence, the main requirement in the class $\mathcal R_+$ is
the uniform bound on the $L^\infty$ norm. As we will see in the next section, this condition is tightly related to
the axiomatization of Regular Lagrangian Flows and it plays an important role, at the technical level.

\smallskip

\subsection{Other measure-theoretic tools}

A measure-valued map $a\mapsto\mu_a\in\Measuresp{X}$ in a measurable space $(\mathcal{A},\mathcal F)$ is said
to be weakly$^*$-measurable if $a\mapsto\mu_a(B)$ is $\mathcal F$-measurable for any Borel set $B\subseteq X$. It is the
typical regularity theory when dealing with conditional probabilities, or disintegration theorems, as illustrated below. 
We shall use the following result:

\begin{theorem}[Disintegration]\label{thm:disinte}
Let $X$, $\mathcal A$ be Polish topological spaces, let $\mathcal F$ be a $\sigma$-algebra containing $\Borel{\mathcal A}$ 
and let $f:X\to\mathcal A$ be a ($\Borel{X}-\mathcal F$)-measurable map. \\
Then, for every nonnegative and $\sigma$-finite measure $\mu$ in $X$  such that $f_\#\mu$ is $\sigma$-finite
there exist a weakly$^*$ measurable family of measures $\mu_a\in \Measuresp{X}$, $a\in\mathcal A$ , concentrated on $f^{-1}(a)$ 
and unique up to sets of negligible $f_\#\mu$-measure, such that
$\mu=\int_{{\mathcal A}}\mu_a\,df_\#\mu$,
namely 
$$
\int_X\phi(x)\,d\mu(x)=\int_{{\mathcal A}}\int_X\phi(x)\,d\mu_a(x)\,d f_\#\mu(a)
$$
for any Borel function $\phi:X\to [0,\infty]$.
\end{theorem}
\begin{proof} By our assumption on $\mathcal F$, the map $f$ is also Borel. Then the result follows by standard
disintegration theorem in Polish (or even Lusin) topological spaces, see for instance \cite[Theorems~10.4.10 and Example~10.4.11]{Bogachev},
after a decomposition of $X$ in countably many Borel sets with finite $\mu$-measure.
\end{proof}

We recall the following basic superposition principle (\cite[Theorem~9]{AC_UMI}), see also Section~\ref{sec:5.2}, where a more general principle
of the same type will play a role.

\begin{theorem}[Superposition principle]\label{thm:superpo}
Under assumption \eqref{eq:growth}, for any weakly continuous solution $\rho_t\in {\mathcal R}_+$ to \eqref{eq:CE} there
exists $\ssigma\in\Measuresp{C([0,T];\R^d)}$ concentrated on absolutely continuous solutions to the ODE \eqref{eq:ODE} and satisfying
\begin{equation}\label{eq:Lag1}
\rho_t\Leb{d}=(\ev_t)_\#\ssigma\qquad \forall t\in [0,T].
\end{equation}
\end{theorem}

The following duality result can be seen as a refinement of Kantorovich's duality in optimal transportation, with the additional information that the
supremum in the dual formulation is attained and that the optimal Kantorovich potentials are characteristic functions if the cost function $c$
is a characteristic function of a Borel set (the latter property is granted by 
the standard construction of the potentials via $c$-monotonicity, see for instance  \cite[Theorem~3.17]{ABS}).

\begin{theorem}[Kellerer duality, \cite{Kel_84}]\label{thm:Kellerer}
Let $X$, $Y$ be Polish spaces and $\mu_X\in\Probabilities{X}$, $\mu_Y\in\Probabilities{Y}$. Let $C\subseteq X\times Y$ be a Borel set. Then
$$
\sup\left\{\Sigma(C):\ \Sigma\in\Gamma(\mu_X,\mu_Y)\right\}=\min\biggl\{\mu_X(A)+\mu_Y(B):\ \uno_A(x)+\uno_B(y)\geq \uno_C(x,y)\biggr\},
$$
where $\Gamma(\mu_X,\mu_Y)$ consists of all $\Sigma\in\Probabilities{X\times Y}$ having $\mu_X$ and $\mu_Y$ as marginals, while
the minimization in the right hand side is among all Borel sets $A\subseteq X$, $B\subseteq Y$.
\end{theorem}

We are not going to use the sharp version, with the $\min$ in the right hand side. Passing to the complementary sets, the non sharp version becomes 
\begin{equation}\label{eq:Ketterer1}
\sup\left\{\mu_X(\tilde{A})+\mu_Y(\tilde{B}):\ (\tilde{A}\times\tilde{B})\cap C=\emptyset\right\}=
2-\sup\left\{\Sigma(C):\ \Sigma\in\Gamma(\mu_X,\mu_Y)\right\}.
\end{equation}

\section{Regular Lagrangian Flows}

\begin{definition} [Regular Lagrangian Flow]
Let $\XX(\cdot,x):\R^d\to C([0,T];\R^d)$ be a Borel map. We say that $\XX$ is a Regular Lagrangian Flow relative to $\bb$ if:
\begin{itemize}
\item[(i)] $\XX(t,x)=x+\int_0^t\bb_s(\XX(s,x))\,ds$ for all $t\in [0,T]$, for $\Leb{d}$-a.e. $x\in\R^d$; 
\item[(ii)] there exists a constant $C>0$ such that $\XX(t,\cdot)_\#\Leb{d}\leq C\Leb{d}$ for all $t\in [0,T]$.
\end{itemize}
The least $C$ in (ii) is called compressibility constant of $\XX$ and denoted by $C(\XX)$. 
\end{definition}

\begin{remark} [Relevance of the bounded compression condition (ii)] 
{\rm Already the case of Sobolev vector fields illustrates the necessity of conditions additional to (i) in order to
recover some form of uniqueness. After the first counterexamples provided in \cite{DL_89} (for ``almost Sobolev'' vector fields, as those in the fractional Sobolev space $W^{s,1}$, $s<1$)
in \cite{Kum_2024}   a continuous, Sobolev and divergence-free vector field $\bb$ is built such that at least two distinct
solutions start from a set of full measure of initial points $x$
(improving an earlier construction in \cite{BCD_2021}, see also \cite{PiSo}). By the uniqueness theorem of RLF, only a suitable selection of these grants the bounded compression (ii), in this specific case 
even the measure-preserving property. The critical integrability exponent of $|\nabla\bb_t|$ for this counterexample is exactly $d$: remarkably,
as soon as $|\nabla\bb_t|\in L^p$ with $p>d$, it is proved in \cite{CC21} that uniqueness in the conventional sense holds for $\Leb{d}$-almost every
initial condition $x$, if one works with the continuous representative of $\bb_t$ granted by the Sobolev embedding.
}\end{remark}

As in the Optimal Transport theory, some statements become extremely more powerful when uniqueness takes place in the
larger class of probability-valued maps, and this leads to the concept of generalized Regular Lagrangian Flow.

\begin{definition} [Generalized RLF]
Let $\R^d\ni x\mapsto\eta_x\in\Probabilities{C([0,T];\R^d)}$ 
be a weakly$^*$ Borel map. We say that $\eta_x$ is a Generalized Regular Lagrangian Flow relative to $\bb$ if 
\begin{itemize}
\item[(i)] for $\Leb{d}$-a.e. $x\in\R^d$, $\eta_x$ is concentrated on absolutely continuous solutions to \eqref{eq:ODE} starting at $x$; 
\item[(ii)] denoting $\eeta=\int\eta_x\,d\Leb{d}(x)\in\Measuresp{C([0,T];\R^d)}$, there exists a constant $C>0$ such that $(\ev_t)_\#\eeta\leq C\Leb{d}$ for all $t\in [0,T]$.
\end{itemize}
The least $C$ in (ii) is called compressibility constant of $\eeta$ and denoted by $C(\eeta)$. 
\end{definition}

Equivalently, thanks to the disintegration theorem (or existence of conditional probabilities, even though in this case the base space $\R^d$ has
$\sigma$-finite measure), a GRLF is $\eeta\in\Measuresp{C([0,T];\R^d)}$, concentrated on absolutely continuous solutions to \eqref{eq:ODE} and
with $(\ev_0)_\#\eeta=\Leb{d}$, such that there exists a constant $C>0$ such that $(\ev_t)_\#\eeta\leq C\Leb{d}$
for all $t\in [0,T]$. Nevertheless, we will often use the disintegrated point of view to clarify our computations. In addition, 
condition (ii) can be formulated as
$$
\int_{C([0,T];\R^d)}\phi(\gamma(t))\,d\eeta(\gamma)=
\int_{\R^d}\int_{C([0,T];\R^d)}\phi(\gamma(t))\,d\eta_x(\gamma)\,dx\leq C\int_{\R^d}\phi(x)\,dx
$$
for any nonnegative Borel $\phi$. 

In the following remark we discuss some cases when existence of a GRLF does not correspond to a RLF.

\begin{remark}
{\rm  A beautiful example in \cite{Pap_2025} shows somehow the necessity to deal with {\it generalized} flows in the absence of sufficient smoothness: for any
 integer $k\geq 1$ and $\alpha\in (0,1)$ a divergence-free vector field $\bb\in C^k_t(C^{0,\alpha}_x)$ is built such that there is no measure-preserving
 selection of characteristics. In this case the superposition principle obviously grants existence of a GRLF, even a measure-preserving one 
 as in Brenier-Schnirelman's theory of generalized incompressible Euler equations (see for instance \cite{AF} and the references therein). 
 More generally, as suggested by the author, without smoothness
 assumptions it is not clear if existence of a RLF for a divergence-free vector field implies that the RLF is measure-preserving.
 
 Still in connection with the impossibility to provide well-behaved selections of characteristics, the well-known example \cite{Depauw_03} of a bounded, divergence-free vector field in $L^1_{\rm loc}\bigl((0,1];BV(\T^d;\R^d)\bigr)$ for which forward uniqueness of \eqref{eq:CE} fails even in the class $\mathcal R_+$
 has been revisited in \cite{Pit_2024}, proving existence and uniqueness of the corresponding GRLF $\eeta$. This GRLF, due to an instantaneous branching at $t=0$, does not have any RLF counterpart, namely a selection of trajectories, indeed the failure of forward uniqueness of \eqref{eq:CE} corresponds to the conditional
 probabilities $\eeta_x$ not being Dirac masses.
}
 \end{remark}

\begin{remark}[Invariance of RLF and GRLF in the Lebesgue equivalence class]\label{rem:invariant}{\rm
The bounded compression condition is useful to show that the notions of RLF and GRLF are invariant in the Lebesgue
equivalence class of $\bb$. Indeed, if $\tilde{\bb}$ is a modification of $\bb$ in a $\Leb{d+1}$-negligible Borel set $M$, we know
that for $\Leb{1}$-a.e. $t\in (0,T)$ the set $\{\gamma:\ (t,\gamma(t))\in M\}$ is $\eeta$-negligible, to be precise, for all $t\in (0,T)$
such that the $t$-section $M_t$ of $M$ is $\Leb{d}$-negligible. Then, Fubini's theorem gives that for $\eeta$-a.e. $\gamma$
the curve $\gamma(t)$ belongs to $M_t$ for a $\Leb{1}$-negligible set of times, whence the invariance of condition (i) follows.}
\end{remark}

\medskip

\begin{remark} [Replacement of $L^\infty$ with $L^p$ bounds on $\rho$]
{\rm In order to have an axiomatization of flows invariant under modifications $\bb$ in $\Leb{d+1}$-negligible sets, the
$L^\infty$ bound could be simply replaced by a non-quantitative bound $\XX(t,\cdot)_\#\Leb{d}\ll\Leb{d}$ for
all (or even $\Leb{1}$-almost all) $t\in (0,T)$, see the argument in Remark~\ref{rem:invariant}. In this connection, on the Eulerian side, in the last few years,
in the class of Sobolev vector fields, many efforts have been devoted to the search of the optimal range of exponents $(p,q)$ (with
$q$ local summability of $\nabla \bb_t$, $p$ local summability of $\rho_t$) ensuring local well
posedness of \eqref{eq:CE}. As soon as $\frac 1p +\frac 1 q<1+\frac 1{dp}$, well posedness in the
class of nonnegative $\rho$'s is ensured by \cite{BCD_2021}, and it is an open problem for signed solutions
(the restriction to nonnegative solutions
comes from the use of the superposition principle, Theorem~\ref{thm:superpo}); on the other hand, the counterexamples built in \cite{BCK}
(see also \cite{MoSa_2020}) show that when $\frac 1p +\frac 1 q>1+\frac 1{dp}$ uniqueness fails.

At the abstract level, it would be interesting to explore a different axiomatization of RLF and GRLF with a quantitative $L^q$ condition,
$q<\infty$, on the time marginals (the main difficulty is due to the fact that $L^q$ bounds on the flow map are not stable under
composition).
}\end{remark}

Even though our main focus in this paper is on the uniqueness for the continuity equation and the self-intersecting properties of the curves
$\XX(t,\cdot)$, in the following remark we recall the known implications at the level of existence, between these two notions of flow.
The questions of uniqueness, or existence {\it and} uniqueness are more subtle and they will be discussed in the next section.

\begin{remark}[Implications at the level of existence]
{\rm Existence of a RLF yields existence of a GRLF, simply setting $\eta_x=\delta_{\XX(\cdot,x)}$.  
On the other hand, existence of a GRLF is equivalent to the existence of nonnegative solutions with $\rho_t$ to \eqref{eq:CE}  
starting from any nonnegative $\rho_0\in L^1\cap L^\infty(\R^d)$ with $\|\rho_t\|_\infty\leq C\|\rho_0\|_\infty$ for some $C$
independent of $\rho_0$: indeed, in one direction it suffices to define $\rho_t$ as density of push forward measures, namely
$$
\rho_t\Leb{d}=(\ev_t)_\#\ssigma\qquad t\in [0,T],
$$
with $\ssigma=\int \rho_0(x)\eta_x\,d\Leb{d}(x)=\rho_0\circ\ev_0\,\eeta$. Conversely, if we take $\rho^R_0=\uno_{B_R(0)}$, the superposition principle grants
$\eeta^R$ satisfying \eqref{eq:Lag1}, so that $(\ev_t)_\#\eeta^R\leq C\Leb{d}$ with $C$ independent of $R$; then, a compactness argument based on the growth condition \eqref{eq:growth} grants the existence
of a $\sigma$-finite positive measure $\eeta$ in $C([0,T];\R^d)$, finite on all sets $\{\gamma:\ \gamma(0)\in B_R(0)\}$ and concentrated on
absolutely continuous solutions to the ODE \eqref{eq:ODE}, with the
property that $(\ev_t)_\#\eeta\leq C\Leb{d}$ for all $t\in [0,T]$. Then, the conditional probability measures $\eeta_x$ provide a GRLF $\eeta$ with
$C(\eeta)\leq C$. We omit the details, as this implication will not be used in the paper.
}\end{remark}

\section{Comparing forward notions of well posedness and nonbranching}\label{sec:nb}

We first state the notion of forward nonbranching, already introduced in \cite{BB21}.

\begin{definition}[Forward nonbranching]\label{deffnb}
Let $\Gamma\subseteq C([0,T];\R^d)$. We say that $\Gamma$ is forward nonbranching if any two different curves $\gamma$, $\tilde\gamma$
in $\Gamma$ either do not intersect in $(0,T]$,
or they coincide after their first positive time of intersection.
\end{definition}

Therefore, according to this definition, the set of positive times in which two distinct curves in a forward nonbranching family intersect can only be
empty or a segment $[t,T]$ with $t>0$. We shall denote by $G$ the forward branching set
\begin{equation}\label{eq:fwd_branching}
G=\left\{(\gamma,\tilde\gamma)\in \bigl(C([0,T];\R^d)\bigr)^2:\ \text{$\gamma\neq\tilde\gamma$, $\gamma(t)=\tilde\gamma(t)$ for some $t\in (0,T)$, 
$\gamma\vert_{[t,T]}\not\equiv\tilde\gamma\vert_{[t,T]}$}\right\}.
\end{equation}

\begin{lemma}\label{lem:Borel}
Let $K\subset\R^d$ be compact and let $\YY:K\to C([0,T];\R^d)$ be continuous and injective. The collection made by all pairs $(x,y)\in K\times K$ 
with $(\YY(\cdot,x),\YY(\cdot,y))\in G$ is a Borel set.
\end{lemma}
\begin{Proof} We can write $G$ as the countable union of the sets
$$G_n:=\left\{(\gamma,\tilde\gamma)\in \bigl(C([0,T];\R^d)\bigr)^2:\ \text{$\gamma(t)=\tilde\gamma(t)$, $\, |\gamma(s)-\tilde{\gamma}(s)| \geq \frac 1n$ for some $s >t \geq \frac 1n$
}\right\},$$
for $n \geq 1$, which are easily seen to be closed. We conclude noticing that the preimage of $G_n$ 
under the continuous map $(x,y) \mapsto (\YY(\cdot,x),\YY(\cdot,y))$ is closed in $K \times K$, therefore the collection in the statement is a countable
union of closed sets.
\end{Proof}

The equivalence principle stated in Theorem~\ref{thm:main} below, established in \cite{Am_2004}, is reminiscent of the general principles relating pathwise
uniqueness to uniqueness in law in probability theory. The difference is due to the selection principle, in the case
of RLF and GRLF granted by the bounded compression condition, in Probability theory (typically) by perturbations
with additive or multiplicative noise. For the sake of completeness, we provide some details about the proof of the
equivalence between (a) and (b) and the consequences (i), (ii) of this equivalence. The new part, in comparison with
\cite{Am_2004}, is the forward nonbranching property (iii). 

\smallskip

\begin{theorem} \label{thm:main} Under assumption \eqref{eq:growth}, the following conditions are equivalent:
\begin{itemize}
\item[(a)] there exists at most one GRLF $\eeta=\int\eta_x\,d\Leb{d}(x)$;
\item[(b)] for any nonnegative initial datum $\rho_0\in L^1\cap L^\infty(\R^d)$, \eqref{eq:CE} has at most one solution $\rho_t\in\mathcal R_+$ starting from $\rho_0$.
\end{itemize}
Moreover, any of these conditions implies that, if there exists a GRLF, then
\begin{itemize}
\item[(i)] $\eta_x$ is a Dirac mass for $\Leb{d}$-a.e. $x\in\R^d$;
\item[(ii)] the unique solution $\rho_t$ to \eqref{eq:CE} in the class ${\mathcal R}_+$ is representable as $(\ev_t)_\#(\int_{\R^d}\rho_0(x)\eeta_x\,dx)$;
\item[(iii)] $\eeta=\int_{\R^d}\eta_x\,dx$ is concentrated on a forward nonbranching Borel family $\Gamma\subseteq C([0,T];\R^d)$.
\end{itemize}
\end{theorem}
\begin{Proof} $(a)\Rightarrow(b)$. Since any solution $\rho_t$ in this class has a representation \eqref{eq:Lag1} via the superposition principle, it suffices to
check that $\ssigma$ is uniquely determined by $\eeta$. Indeed, we can represent $\ssigma$ as $\int\rho_0(x)\sigma_x\,d\Leb{d}(x)$, with $\sigma_x$ concentrated
probability measures on absolutely continuous solutions to the ODE \eqref{eq:ODE} starting at $x$. As $\|\rho_t\|_\infty$ is uniformly bounded, one also has
$(\ev_t)_\#\ssigma\leq C\Leb{d}$ for all $t\in [0,T]$. Now, assume by contradiction that $\sigma_x\neq\eta_x$ in a set with positive $\rho_0\Leb{d}$-measure, 
so that there exists $\tau>0$ such that $\Leb{d}(F_\tau)>0$, with
$$
F_\tau=\{x\in\R^d: \rho_0(x)\geq\tau,\,\,\sigma_x\neq\eta_x\}.
$$
If we define
$$
\tilde\eta_x:=
\begin{cases}
\eta_x & \text{if $x\in\R^d\setminus F_\tau$};
\\
\sigma_x & \text{if $x\in F_\tau$}
\end{cases}
$$
we obtain a different GRLF $\tilde\eta$ with $C(\tilde{\eeta})\leq C(\eeta)+C/\tau$, a contradiction. This proves uniqueness of $\eeta$
and the representation of solutions to \eqref{eq:CE} in (ii). 
\smallskip
 
$(b)\Rightarrow(a)$. See Theorem~5.4 of \cite{Am_2004} (or Theorem~3.10 of \cite{AmCri_2014}). The strategy of the proof is to show that any GRLF $\eta_x$ is a Dirac
mass for $\Leb{d}$-a.e. $x\in\R^d$ (thus proving (i) as well), arguing by contradiction and, in particular, using the set of points $x$ where this is not the case to build two
distinct solutions to \eqref{eq:CE} in the class $\mathcal R_+$ and with the same initial condition.

\smallskip

Now, let us prove (iii). Since $\eta_x=\delta_{\XX(\cdot,x)}$ for $\Leb{d}$-a.e. $x\in\R^d$, it suffices to show the existence of a 
 $\Leb{d}$-negligible Borel set $\mathcal N\subseteq\R^d$ such that the family of curves $\{\XX(\cdot,x)\}_{x\in\R^d\setminus\mathcal N}$ 
is forward nonbranching. Let now $\mu_d$ be the standard Gaussian in $\R^d$ (or any other probability measure mutually absolutely 
continuous with $\Leb{d}$); it suffices to show the existence for
any $\epsilon>0$ of a Borel set $\mathcal N_\epsilon\subseteq\R^d$ such that  $\mu_d(\mathcal N_\epsilon)<2\epsilon$ and such that
the family of curves $\{\XX(\cdot,x)\}_{x\in\R^d\setminus\mathcal N_\epsilon}$ is forward nonbranching, as the $\mu_d$-negligible set
$\mathcal N=\limsup_{k\to\infty}\mathcal N_{2^{-k}}$ gives the result.

Now, fix $\epsilon>0$ and a compact set $K\subset\R^d$ with $\mu_d(\R^d\setminus K)<\epsilon$ and
$x\mapsto\XX(\cdot,x)$ is continuous. Following a strategy similar to \cite{Bia_Bo_2020}, we will apply Theorem~\ref{thm:Kellerer} with $X=Y=\R^d$, $\mu_X=\mu_Y=\mu_d$ 
and
$$
C=\left\{(x,y)\in K\times K: (\XX(\cdot,x),\XX(\cdot,y))\in G\right\},
$$
where $G$ is the forward branching set in \eqref{eq:fwd_branching}.
Thanks to Lemma~\ref{lem:Borel} and the continuity of $x\mapsto\XX(\cdot,x)$ on $K$, $C$ is a Borel set. The ``switching'' operator
$$
{\sf Sw}(\gamma,\tilde\gamma)(s):=
\begin{cases}
\gamma(s) &\text{if $s\in [0,t_{(\gamma,\tilde\gamma)}]$};\cr\cr
\tilde\gamma(s) &\text{if $s\in [t_{(\gamma,\tilde\gamma)},T]$},
\end{cases}
$$
where $t_{(\gamma,\tilde{\gamma})}>0$ is the first time of intersection, is well defined on $G$ and
maps $(\gamma,\tilde\gamma)$ to a curve ${\sf Sw}(\gamma,\tilde\gamma)\neq\gamma$. 

Theorem~\ref{thm:Kellerer} and \eqref{eq:Ketterer1} grant that
either there exists $\Sigma\in\Probabilities{\R^d\times\R^d}$ having $\mu_d$ as marginals with $\Sigma(C)>\epsilon$, or there exist Borel sets $A,\,B\subseteq\R^d$
with $\mu_d(A)+\mu_d(B)>2-\epsilon$ and $(A\times B)\cap C=\emptyset$. If the second possibility occurs, we may take
$\mathcal N_\epsilon=\R^d\setminus (K\cap A\cap B)$, as
$$
\mu_d(A\cap B)\geq\mu_d(A)+\mu_d(B)-\mu_d(A\cup B)>1-\epsilon.
$$
To conclude the proof, it suffices to show that the first possibility does not occur. If such $\Sigma$ exists, write $\hat{\Sigma}=\uno_C\Sigma$
and $\hat{\Sigma}=\int_{\R^d}\hat{\Sigma}_x\,d\mu(x)$, with $\mu=\rho\Leb{d}$, the marginal of $\hat{\Sigma}$ with respect to the first factor
and $\hat{\Sigma}_x\in\Probabilities{K}$. As $\mu(\R^d)=\Sigma(C)>0$, we have 
$\Leb{d}(F_\tau)>0$ for some $\tau>0$, with $F_\tau=\{\rho>\tau\}$. Now, we may replace $\eta_x=\delta_{\XX(\cdot,x)}$ by 
$$
\tilde\eta_x:=
\begin{cases}
{\sf Sw}(\XX(\cdot,x),\XX(\cdot,y))_\#\hat{\Sigma}_x &\text{if $x\in F_\tau$;}\cr\cr
\eta_x & \text{if $x\in\R^d\setminus F_\tau$.}
\end{cases}
$$
It is clear that $\tilde\eta_x\neq\eta_x$, as $\Leb{d}(F_\tau)>0$ and 
${\sf Sw}(\gamma,\tilde\gamma)\neq\gamma$ whenever $(\gamma,\tilde{\gamma})\in G$.
It is not difficult to show that still $\tilde\eeta$ has bounded compression, whence the contradiction. 

Let us prove, indeed, that $\tilde\eeta$ has bounded compression; for any Borel set $E\subseteq\R^d$ and any $t\in [0,T]$
we can estimate from above 
$$
\tilde\eeta\bigl(\{\gamma:\ \gamma(t)\in E\}\bigr)\leq C(\eeta)\Leb{d}(E)+
\int_{F_\tau}\hat{\Sigma}_x\bigl(\{y:\ {\sf Sw}(\XX(\cdot,x),\XX(\cdot,y))(t)\in E\}\bigr)\,dx.
$$
By the very definition of ${\sf Sw}$, the second addendum in the right hand side can be estimated from above by
\begin{equation}
\Leb{d}\bigl(\{x\in F_\tau:\ \XX(t,x)\in E\}\bigr)+\int_{F_\tau}\hat{\Sigma}_x\bigl(\{y:\ \XX(t,y)\in E\}\bigr)\,dx.
\end{equation}
The first term can be estimated once more with $C(\eeta)\Leb{d}(E)$; the second term, recalling that
$\mu_d\leq C_d\Leb{d}$ and $\hat{\Sigma}\leq\Sigma$, by
$$
\frac{1}{\tau}\int_{\R^d}\hat{\Sigma}_x\bigl(\{y:\ \XX(t,y)\in E\}\bigr)\,d\mu_d(x)\leq
\frac{1}{\tau}\mu_d(\{y:\XX(t,y)\in E\})\leq\frac{C_d}{\tau}C(\eeta)\Leb{d}(E),
$$
concluding the proof.
\end{Proof}

\begin{remark}[Some open problems] {\rm 
It would be interesting to know to what extent the equivalence between (a) and (b) in Theorem~\ref{thm:main} can be sharpened. For instance, what are the implications
at the level of the continuity equation \eqref{eq:CE} if we only assume existence and uniqueness of the RLF, not assuming anything on the GRLF? While
existence in the class ${\mathcal R}$ is obvious, the superposition principle of Theorem~\ref{thm:superpo} does not seem to be sufficient to grant uniqueness, even
in ${\mathcal R}_+$, this is an open question. In this connection, in \cite{Bo_Gu}, a bounded vector field $\bb$ is built in one space dimension
whose characteristics are constant: for this vector field existence and uniqueness even of the GRLF are trivial, as well as existence and uniqueness of 
nonnegative measure-valued solutions and solutions in the class ${\mathcal R}_+$ (namely, the constant ones), while uniqueness is proved to fail for measure-valued solutions.  It would be interesting to know if uniqueness for \eqref{eq:CE} might fail even in the class ${\mathcal R}$.
 }\end{remark}

\begin{corollary}[Enhanced uniqueness of the RLF]\label{cor:main}
Assume  that any of the equivalent conditions in Theorem~\ref{thm:main} holds  and that there exists a GRLF. Then:
\begin{itemize}
\item[(a)] the RLF $\XX$ exists and is unique, and $\eeta_x=\delta_{\XX(\cdot,x)}$ for $\Leb{d}$-a.e. $x\in\R^d$;
\item[(b)] there exists a $\Leb{d}$-negligible Borel set $\mathcal N\subseteq\R^d$ such that 
the family of curves $\{\XX(\cdot,x)\}_{x\in\R^d\setminus\mathcal N}$ is forward nonbranching;
\item[(c)] If the same conditions
hold for $\tilde{\bb}_t:=-\bb_{T-t}$ then $\Leb{d}\ll (\ev_T)_\#\eeta=\XX(T,\cdot)_\#\Leb{d}$, 
and there exists a $\Leb{d}$-negligible Borel set $\mathcal N\subseteq\R^d$ such that 
the family of curves $\{\XX(\cdot,x)\}_{x\in\R^d\setminus\mathcal N}$ is disjoint in $[0,T]$.
\end{itemize}
\end{corollary}

\begin{Proof} Statements (a) and (b) are an immediate consequence of Theorem~\ref{thm:main}. In order to prove (c), if
we call $\tilde\eeta=\int\tilde\eeta_y\,d\Leb{d}(y)$ the unique GRLF relative to $\tilde\bb$, it is clear that $\eeta'=j_\#\tilde\eeta$ yields a solution of \eqref{eq:CE} with vector field $\bb$,
where 
$$
j:C([0,T];\R^d)\to C([0,T];\R^d)\quad\text{is the time inversion, $j(\gamma)(t):=\gamma(T-t)$},
$$
as the bounded compression property at any time $t\in [0,T]$ and the property of being concentrated on solutions are preserved by this transformation.
In particular ${\eeta}'=\tilde\rho_T\mathcal{L}^{d}(y)\otimes{\eeta}_y$, where $\ev_{T\#}\tilde\eeta=\tilde\rho_T\mathcal{L}^{d}$.
This implies that for $\Leb{d}$-a.e. $y$ the measure $\tilde\eeta_y$ is concentrated on
the curves $\XX(\cdot,x)$ such that $y=\XX(T,x)$, as $\eeta$ is concentrated on this family of curves.  Since, in addition, 
$\tilde\eeta_y=\delta_{\YY(\cdot,y)}$ has to be a Dirac mass for $\Leb{d}$-a.e. $y\in\R^d$, 
we obtain the existence of a Borel set $Y_0\subseteq\R^d$ of full $\Leb{d}$-measure such that for all $y\in Y_0$ there
exists only one $x(y)$ such that $y=\XX(T,x(y))$ and more generally $\YY(T-t,y)=\XX(t,x(y))$ for all $t\in [0,T]$.  

As $x(y)=\YY(T,y)$, $y\in Y_0$, thanks to Lusin's theorem (applied for instance with the finite reference measure $\gamma_d$, mutually absolutely
continuous with $\Leb{d}$)
we can assume with no loss of generality that $Y_0$ is the union of compact sets $K_i$ such that $\YY(T,\cdot)$ is 
continuous on $K_i$; it follows that $X_0=\YY(T,Y_0)$ is a Borel set. Now we can prove that $\Leb{d}\ll \XX(T,\cdot)_\#\Leb{d}$,
more precisely $\XX(T,\cdot)_\#\Leb{d}\geq (C(\tilde{\eeta}))^{-1}\Leb{d}$. Indeed, for any compact set $E\subset\R^d$ one has
\begin{eqnarray*}
\Leb{d}\bigl(\{x\in\R^d:\ \XX(T,x)\in E\}\bigr)&\geq& \Leb{d}\bigl(\{x\in X_0:\ \XX(T,x)\in E\}\bigr)\\
&=&\Leb{d}\bigl(\YY(T,\cdot)(Y_0\cap E)\bigr)\\
&\geq&\frac 1{C(\tilde{\eeta})}\Leb{d}\bigl(\{y:\ \YY(T,y)\in\YY(T,\cdot)(Y_0\cap E)\}\bigr)\\&\geq&\frac 1{C(\tilde{\eeta})}\Leb{d}(Y_0\cap E)\bigr)
=\frac 1{C(\tilde{\eeta})}\Leb{d}(E).
\end{eqnarray*}
Eventually, by repeating the argument symmetrically, we may take as $\mathcal N$ the union $\mathcal N_1\cup\mathcal N_2$, where $\mathcal N_1$ and $\mathcal N_2$ are
the Lebesgue negligible sets associated to $\XX$ and $\YY$ respectively.
\end{Proof}

In particular, the assumption in (c) is satisfied under the assumptions in \cite{DL_89}, dealing with Sobolev vector fields, and
\cite{Am_2004}, dealing with $BV$ vector fields, in \cite{DL_89} assuming $L^1_tL^\infty_x$ bounds on 
${\rm div\,}\bb_t$ and in \cite{Am_2004} relaxing this assumption to lower $L^1_tL^\infty_x$ bounds on 
${\rm div\,}\bb_t$, but still assuming it to be absolutely continuous with respect to $\Leb{d}$. 

A further relaxation of the bounds on divergence is the {\it near incompressibility}, 
an assumption singled out in \cite{Am_Del_Maly_2007} which amounts to a sort of integral $L^\infty$ condition on the divergence.
Its one-sided version, where only upper bounds on $\rho$ are required, will be considered in the final part of Section~\ref{sec:5.2}, see Definition~\ref{def1sidednearlyincompressible}.

\begin{definition}[Nearly incompressible vector field]\label{def:nivf}
We say that $\bb$ is nearly incompressible if there exists a Borel positive function $\rho$ with 
$\rho+\displaystyle{\frac 1\rho}\in L^\infty\bigl((0,T)\times\R^d\bigr)$ such that $\rho_t=\rho(t,\cdot)$ solve
\eqref{eq:CE} in the sense of distributions.
\end{definition}

In \cite{Bia_Bo_2020}  the well posedness result for nonnegative solutions has been extended assuming near incompressibility
of the vector field (after partial results in \cite{Am_Del_Maly_2007}), see the final part of Section~\ref{sec:5.2} for a more detailed
discussion.
See also \cite{ABD_04}, \cite{Bre_2003} for a natural context where a vector field with divergence not absolutely continuous with respect to
$\Leb{d}$, but with bounded compression appears.

\section{Forward well posedness in $SBD$}\label{sec:sbd}

In this section we present a forward well-posedness result in a class of space-time $BD$ vector fields whose spatial components
are $SBD$. We also allow $D_x\cdot\bb_t$ to be a measure, requiring only the lower bound \eqref{eq:onesideddiv}.

\begin{theorem}\label{thm:main2}
Let $\bb\in L^\infty_{{\rm loc}}\bigl((0,T)\times\R^d;\R^d\bigr)$ be satisfying \eqref{eq:growth}, with 
\begin{equation}\label{eq:BDassumption}
\|\bb_t\|_{SBD(B_R(0))}\in L^1(0,T)\qquad\text{for all $R>0$,}
\end{equation}
\begin{equation}\label{eq:onesideddiv}
D_x\cdot\bb_t\geq -\xi(t)\Leb{d} \quad\text{with $\xi\in L^1(0,T)$,} 
\end{equation}
\begin{equation}\label{eq:time_regularity}
\text{$\partial_t\bb_t$ is a $\R^d$-valued locally finite measure in $[0,T)\times\R^d$.}
\end{equation}
Then \eqref{eq:CE} has existence and uniqueness in the class ${\mathcal R}_+$. In particular,
under these assumptions, the GRLF associated to $\bb$ is unique and therefore Theorem~\ref{thm:main} and Corollary~\ref{cor:main}(a), (b)
apply to $\bb$.
\end{theorem}

Theorem~\ref{thm:main2} differs from \cite{Am_Del_Maly_2007} (where already the case of space-time $SBD$ vector fields $\BB$ with
bounded compression was covered). The regularity assumption on $\bb$ is milder and the assumption on the divergence of $\bb$ is relaxed
to \eqref{eq:onesideddiv}, therefore uniqueness is obtained only in the class ${\mathcal R}_+$. The $BD_{\rm loc}\bigl((0,T)\times\R^d\bigr)$ regularity of the vector field $\BB=(1,\bb)$ in \eqref{eq:spacetime} is a simple consequence of the combination of
\eqref{eq:BDassumption} and \eqref{eq:time_regularity}: it
is needed in the next proposition to ensure that $\Leb{1}\otimes E^s\bb_t$ is concentrated on a countably $\Haus{d}$-rectifiable set, even though
\eqref{eq:BDassumption} and \eqref{eq:time_regularity} do not imply that $\BB\in SBD_{\rm loc}$. Comparing with the results of the next section,
where no regularity of $\bb_t$ with respect to $t$ will be assumed, it would be interesting to remove \eqref{eq:time_regularity}.

We provide a proof via the chain rule for traces (\cite{Am_Cri_Man_2005,Am_Del_Maly_2007}). 
The non trivial part of Theorem~\ref{thm:main2} is uniqueness. Indeed, existence in the class ${\mathcal R}_+$ (even of signed solutions,
starting from any $\rho_0\in L^1\cap L^\infty$)  comes from a simple smoothing argument based on
\eqref{eq:onesideddiv} and from the a priori bounds (see for
instance \cite{AmCri_2014})
\begin{equation}\label{eq:apriori1}
\|\rho_t\|_\infty\leq\|\rho_0\|_\infty\exp\biggl(\int_0^t \|({\rm div\,}\bb_s)^-\|_\infty\,ds\biggr),\qquad
\|\rho_t\|_1\leq \|\rho_0\|_1,
\end{equation}
valid for smooth vector fields. We omit the details.

On the other hand, as Example~\ref{example:nogood} shows, for uniqueness the nonnegativity assumption may be necessary if we allow
the divergence to be singular. However, well posedness within nonnegative solutions is sufficient
to apply to this class of vector fields $\bb$ the general results of Section~\ref{sec:nb} dealing with the RLF.

\begin{proposition}\label{precti}
Under the assumptions of Theorem~\ref{thm:main2} one has
\begin{equation}\label{eq:onesideddiv1}
\trplus{\BB}{\nu_\Sigma}+\trplus{\BB}{-\nu_\Sigma}\geq 0 \qquad\text{$\Haus{d}$-a.e. on $\Sigma$}
\end{equation}
for any countably $\Haus{d}$-rectifiable set $\Sigma$. \\
In addition, $\Leb{1}\mid_{[0,T]}\otimes |E^s\bb_t|$ is concentrated on the countably $\Haus{d}$-rectifiable set $J_{\BB}$.
\end{proposition}
\begin{Proof}  We canonically extend $\BB$ to $(-1,T)\times\R^d$. 
From  \eqref{eq:BDassumption} and \eqref{eq:time_regularity} we immediately obtain that $\BB\in BD_{\rm loc}\bigl((-1,T)\times\R^d\bigr)$
and, in particular, $\BB\in {\mathcal M}_\infty\bigl((-1,T)\times\R^d\bigr)$. Then, \eqref{eq:onesideddiv1} follows by Proposition~\ref{pbasic}.

Recall that, by the $SBD$ assumption on $\bb_t$, $E^s\bb_t$ is concentrated on $J_{\bb_t}$ for $\Leb{1}$-a.e. $t\in (0,T)$. It is obviously
sufficient to show that $ \Leb{1}\mid_{[0,T]}\otimes |\langle E^s\bb_t\eta,\eta\rangle|\, $ is concentrated on the countably $\Haus{d}$-rectifiable set $J_{\BB}$ for
any unit vector $\eta\in\R^d$. By the representation of $|\langle E^s\bb_t\eta,\eta\rangle|$ provided by 
Proposition~\ref{structural} it is sufficient to show that for $\Leb{1}$-a.e. $t\in (0,T)$ one has
$$
\left\{x= (t,y+s\eta):\ y\in\pi_\eta,\, s\in J (\bb_t)_{y,\eta}\right\} \subseteq J_{\BB}
\quad\text{for $\Haus{d-1}$-a.e. $y\in\pi_\eta$.}
$$
Taking into account that $\BB=(1,\bb)$, this is a direct consequence of the application of \eqref{eq:sliceJ2} of Proposition~\ref{structural} to $\BB$
with $\xi=(0,\eta)$, as 
$$
(\BB)_{z,\xi}(s)=\langle\bb_t(y+s\eta),\eta\rangle=(\bb_t)_{y,\eta}(s)
\qquad\text{for any $z=(t,y)\in\pi_\xi\subseteq\R^{d+1}$}.
$$
\end{Proof}

\begin{Proof}[of Theorem~\ref{thm:main2}]
In this proof we canonically extend to times $t\in (0,T)^c$ both $\bb_t$ and the solutions $\rho_t$ to \eqref{eq:CE_initial} as
described in Remark~\ref{rem_negative_times}.
Under the assumption \eqref{eq:BDassumption}, if $\rho\in\mathcal R$ solves \eqref{eq:CE_initial} 
it is clear that $\rho \BB\in {\mathcal M}_\infty\bigl(\R\times\R^d\bigr)$ with 
\begin{equation}\label{eq:rep_div_BB_1}
D\cdot (\rho\BB)=\bar\rho\Haus{d}\res\{t=0\}-\rho_T\Haus{d}\res\{t=T\}
\qquad\text{in $\R\times\R^d$.}
\end{equation}
In particular Proposition~\ref{pbasic}  gives 
\begin{equation}\label{eq:agreement}
\trplus{\rho \BB}{\nu_\Sigma}+\trplus{\rho \BB}{-\nu_\Sigma}=0\qquad
\text{$\Haus{d}$-a.e. on $\Sigma$}
\end{equation}
for any countably $\Haus{d}$-rectifiable set $\Sigma\subseteq (0,T)\times\R^d$. In addition, \eqref{goodac} gives:
\begin{equation*}
D^a\cdot (h(\rho)\BB)=(h(\rho)-\rho h'(\rho))D^a\cdot \BB\qquad\forall h\in C^1(\R).
\end{equation*}
Choosing $h_\epsilon(z)=\sqrt{\epsilon^2+(z^+)^2}-\epsilon\in C^1(\R)$ and using the uniform control on the singular parts 
of $D \cdot (h_\epsilon(\rho)\BB)$ provided by
\eqref{eq:verygoodac}, by approximation we obtain
\begin{equation}\label{eq:noac}
D^a\cdot (\rho^+\BB)=0.
\end{equation}

On the other hand, for any countably $\Haus{d}$-rectifiable set $\Sigma\subseteq (0,T)\times\R^d$, the chain rule for traces and Proposition~\ref{pbasic} give that for any $h\in C^1(\R)$ and $R>0$, if we set $\Sigma_R:=\Sigma \cap ((0,T) \times B_R(0))$,

\begin{eqnarray}\label{eq:goodsigma1}
&&D^s\cdot (h(\rho)\BB)(\Sigma_R)=\int_{\Sigma_R }(\trplus{h(\rho)\BB}{\nu_\Sigma}+\trplus{h(\rho)\BB}{-\nu_\Sigma})\,d\Haus{d}=\\
&&\int_{\Sigma_R} \trplus{\BB}{\nu_\Sigma}  h\left(\frac{\trplus{\rho\BB}{\nu_\Sigma}}{\trplus{\BB}{\nu_\Sigma}}\right)+
\trplus{\BB}{-\nu_\Sigma}  h\left(\frac{\trplus{\rho\BB}{-\nu_\Sigma}}{\trplus{\BB}{-\nu_\Sigma}}\right)\,d\Haus{d}.\nonumber
\end{eqnarray}
Passing to the limit with the same approximating sequence $h_\eps$ as before, we can set $h(z)=z^+$, notice that \eqref{eq:agreement} shows that points where $\trplus{\BB}{\nu_\Sigma}$ and
$\trplus{\BB}{-\nu_\Sigma}$ have different sign give no contribution to the integral. Because of \eqref{eq:onesideddiv1}, the case when both
are negative can be ruled out, so that 
\begin{equation}\label{eq:goodsigma2}
D^s\cdot (\rho^+\BB)(\Sigma_R)=\int_{\Sigma_R \cap \{\trplus{\BB}{\nu_\Sigma},\,\trplus{\BB}{-\nu_\Sigma}\geq 0\}} |\trplus{\rho\BB}{\pm\nu_\Sigma}|\,d\Haus{d}.
\end{equation}

Now, Proposition~\ref{precti} grants that $\Leb{1}\mid_{[0,T]}\otimes |E^s\bb_t|$ is concentrated on $J_\BB$ and \eqref{eq:verygoodac} grants that the same
is true for $D^s\cdot (\rho^+\BB)$. Hence, we may choose as $\Sigma$ the set $J_\BB\cap (0,T)\times\R^d$ to obtain from \eqref{eq:noac}, \eqref{eq:goodsigma2}
and Lemma~\ref{lem:cylindrical} (choosing $R=R_i$ and sending $i \to \infty$)
\begin{equation}\label{eq:goodsigma3}
\int_{\R^d}\rho_T^+\,dx-\int_{\R^d}\bar\rho^+\,dx=\int_{\{\trplus{\BB}{\nu_\Sigma},\, \trplus{\BB}{-\nu_\Sigma}\geq 0\}
\cap (0,T)\times\R^d}|\trplus{\rho\BB}{\pm\nu_\Sigma}|\,d\Haus{d}.
\end{equation}
In particular, if $\rho\in {\mathcal R}_+$, by mass conservation the right hand side must vanish. Hence, we may apply
\eqref{eq:goodsigma2} to the difference $\rho=\rho^1-\rho^2$, with $\rho^i\in {\mathcal R}_+$ solving \eqref{eq:CE_initial} with
the same initial datum $\bar\rho$, to obtain from Lemma~\ref{lem:cylindrical} again
\begin{eqnarray*}
\int (\rho^1_T-\rho^2_T)^+\,dx&=&\int_{\{\trplus{\BB}{\nu_\Sigma},\,\trplus{\BB}{-\nu_\Sigma}\geq 0\}\cap (0,T)\times\R^d}
\hskip -2cm |\trplus{(\rho^1-\rho^2)\BB}{\pm\nu_\Sigma}|\,d\Haus{d}\\
&\leq&
\int_{\{\trplus{\BB}{\nu_\Sigma},\,\trplus{\BB}{-\nu_\Sigma}\geq 0\}\cap (0,T)\times\R^d}
\hskip -2cm |\trplus{\rho^1\BB}{\pm\nu_\Sigma}|+|\trplus{\rho^2\BB}{\pm\nu_\Sigma}|\,d\Haus{d}=0.
\end{eqnarray*}
Exchanging the roles of $\rho^1$ and $\rho^2$ and repeating this argument with any intermediate time $t\in (0,T)$ gives the result.
\end{Proof}

\begin{lemma}\label{lem:cylindrical}
If $\rho\in\mathcal R_+$ satisfies $D\cdot (\rho\BB)=\mu$ with $\mu$ finite in bounded subsets of $(0,T)\times\R^d$, 
then there exist $R_i\to \infty$ such that
$$
\int_{\R^d}\rho_T\,dx-\int_{\R^d}\bar\rho\,dx=\lim_{i\to\infty} D\cdot (\rho\BB)\bigl((0,T)\times B_{R_i}\bigr).
$$
\end{lemma}
\begin{proof} Let $C=\sup_{s\in [0,T]}\|\rho_s\|_\infty$.
From \eqref{eq:rep_div_BB}, for all $R>0$ one has
\begin{equation}\label{eq:seraop}
\int_{B_R}\rho_T\,dx-\int_{B_R}\bar\rho\,dx-D\cdot (\rho\BB)\bigl((0,T)\times B_R\bigr)=
\int_{(0,T)\times\partial B_R}\trplus{\rho \BB}{\nu_{(0,T)\times\partial B_R}}\,d\Haus{d}.
\end{equation}
where the normal is chosen to be the inner one to the cylinder $(0,T)\times B_R$.
Now, the growth condition \eqref{eq:growth} ensures 

\begin{equation*}
\begin{split}
\lim_{i\to\infty}&2^{2-i}\int_{2^i}^{2^{i+1}}\int_{(0,T)\times\partial B_R}\rho|\BB|\,d\Haus{d}\,dR\leq\lim_{i\to\infty}\int_{2^i}^{2^{i+1}}\int_{(0,T)\times\partial B_R}\rho\frac{|\BB|}{1+|x|}\,d\Haus{d}\,dR\\
&\leq \lim_{i\to\infty} \int_{2^i}^{2^{i+1}}\int_{(0,T)\times\partial B_R}\rho\,d\Haus{d}\,dR+ \int_{2^i}^{2^{i+1}}\int_{(0,T)\times\partial B_R}\rho\frac{|\bb|}{1+|x|}\,d\Haus{d}\,dR\\
&\leq\lim_{i\to\infty}  \|\rho\|_{\infty}\int_{0}^{T}\int_{B_{2^{i+1}}\setminus B_{2^{i}}}\frac{|\bb^1|}{1+|x|}\,dx\,ds +\int_{0}^{T}\|\frac{|\bb_s^2|}{1+|x|}\|_{\infty}\int_{B_{2^{i+1}}\setminus B_{2^{i}}}\rho\,dx\,ds =0
\end{split}
\end{equation*}
where the last term goes to 0 by dominated convergence.
Hence, we can find radii $R_i\in (2^i,2^{i+1})$ with $\int_{(0,T)\times\partial B_{R_i}}{\rho|\BB|}\,d\Haus{d}\to 0$
and $|\trplus{\rho \BB}{\nu_{(0,T)\times\partial B_{R_i}}}|\leq \rho|\BB|$ $\Haus{d}$-a.e. in $(0,T)\times\partial B_{R_i}$
to conclude that the term in the right-hand side of \eqref{eq:seraop} with $R=R_i$ is infinitesimal as $i\to\infty$.
\end{proof}

\section{The general BV case}\label{sec:5.2}

We preliminarily revisit some results and notions from \cite{Bia_Bo_2020}, regarding Lagrangian representations,
the untanglement property and the partition via characteristics. Throughout this section, unless differently specified, we shall always assume that
$\bb:(0,T)\times\R^d\to\R^d$ satisfies \eqref{eq:growth} and consider densities $\rho\in \mathcal{R}^\mathcal{I}_+:=\{\rho_t\in L^\infty\bigl(\mathcal{I};L^1(\R^d)\bigr)\cap L^\infty\bigl(\mathcal{I}; L^\infty(\R^d)\bigr):
\,\rho_t \,\mbox{is nonnegative}\}$, where $\mathcal{I}$ is some bounded open interval. We shall canonically extend $\bb$ and functions in  $\mathcal{R}^\mathcal{I}_+$  to $0$ outside of their domains of definition.

Recall
also the equivalent formulations of the continuity equations in Remark~\ref{rem_negative_times}, possibly involving a right hand
side.

\subsection{Spaces of curves and Lagrangian representation}

In this section only we shall adopt a different notation for curves. More precisely,
$\gamma$ will denote a $\R^d$-valued uniformly continuous curve from a bounded, nonempty open interval $I_\gamma\subseteq\R$, and set $I_\gamma=(t_\gamma^-,t_\gamma^+)$. Since $\gamma$ is uniformly continuous, it has one-sided limits
$\gamma(t_{\gamma+}^-)$, $\gamma(t_{\gamma-}^+)$ at the extreme points of $I_\gamma$ and we denote by
$$
\overline{\mathrm{Graph\,}}\gamma=
\left\{(t,\gamma(t)):\ t\in I_\gamma\right\}\cup\left\{(t_\gamma^-,\gamma(t_{\gamma+}^-)),(t_\gamma^+,\gamma(t_{\gamma-}^+))\right\}
$$
the completion of $\mathrm{Graph}\,\gamma$, the graph of $\gamma$ on $I_\gamma$.

Denoting by $\Upsilon$ the collection of these curves, it can canonically be given a Polish structure,
by considering the Hausdorff distance $d_\Upsilon$ between the corresponding completed graphs; according to $d_\Upsilon$, convergent
sequences $(\gamma_h)$ are equicontinuous and
$$
\lim_{h\to\infty} t_{\gamma_h}^-=t_\gamma^-,\quad
\lim_{h\to\infty} t_{\gamma_h}^+=t_\gamma^+,\quad
\gamma_h\to\gamma\,\,\,\text{uniformly on compact subsets of $(t_\gamma^-,t_\gamma^+)$.}
$$
for the limit curve $\gamma$.

It will also be convenient to consider the class $\Upsilon_0$ of locally uniformly  continuous curves
$\sigma:(t_\sigma^-,t_\sigma^+)\to\R^d$. For these curves $\sigma$ we
call $\sigma^k\in\Upsilon$ the restriction of $\sigma$ to the subdomain
\begin{equation}\label{eq:k-restriction}
I_{\sigma,k}:=\left\{t\in\R:\ |t|<k,\,\,\,t\in \bigl(\inf I_\sigma+\frac{1}{k}\ell_\sigma,\sup I_\sigma-\frac{1}{k}\ell_\sigma\bigr)\right\}
\qquad k\geq 2
\end{equation}
with $\ell_\sigma=(t_\sigma^+-t_\sigma^-)>0$. As all curves $\sigma^k$ belong to $\Upsilon$, a Polish
structure in $\Upsilon_0$ is the coarsest topology that makes all maps $\sigma\mapsto\sigma^k$, $k\geq 2$,
continuous. Also, according to this definition, the map which associates to $\gamma\in\Upsilon$ the
corresponding curve $\sigma\in\Upsilon_0$ is continuous.

We denote also
$$
\Gamma_\bb=\left\{\gamma\in \Upsilon:\ \gamma \text{ is a characteristic curve of $\bb$}\right\},
$$
meaning as in the previous sections that $\gamma$ is absolutely continuous in $\overline{I}_\gamma$
and an integral solution of the ODE $\gamma'(t)=\bb_t(\gamma(t))$. It is easily seen that
$\Gamma_\bb$ is a Borel subset of $\Upsilon$. Analogously,
$$
\Gamma_{0,\bb}=\left\{\gamma\in \Upsilon_0:\ \gamma \text{ is a characteristic curve of $\bb$}\right\},
$$
 is a Borel subset of $\Upsilon_0$ containing locally absolutely continuous curves that are solutions of the ODE.

\begin{definition}[Lagrangian Representation]\label{defLagRep}
Assume that the vector field $\rho\BB=(\rho,\rho\bb)$ satisfies $D\cdot (\rho\BB)=\mu\in\Measures{\R^{d+1}}$.
A Lagrangian representation of $\rho\BB$ is a measure $\eeta\in\Measuresp{\Upsilon}$ 
concentrated on $\Gamma_\bb$ such that
\begin{equation}\label{rhoBBrep}
\rho\BB\Leb{d+1}=\int_\Upsilon(\mathrm{Id}, \gamma)_\#\left((1,\gamma') \Leb{1}\right)\,d\eeta(\gamma) 
\end{equation}
and
\begin{align}
\mu & =\int_\Upsilon\left[\delta_{(t_\gamma^-,\gamma(t_\gamma^-))}(d t, d x)-\delta_{(t_\gamma^+,\gamma(t_\gamma^+))}(d t, d x)\right] \, d\eeta(\gamma), \label{eq:concancellazioni}\\
|\mu| & =\int_\Upsilon\left[\delta_{(t_\gamma^-,\gamma(t_\gamma^-))}(d t, d x)+\delta_{(t_\gamma^+, \gamma(t_\gamma^+))}(d t, d x)\right] \, d\eeta(\gamma).\label{eq:nocancellazioni}
\end{align}
\end{definition}

The existence of Lagrangian representations in this context, connecting the Eulerian and Lagrangian points of view, 
was first proved in \cite{Am_2004}. Alternatively, existence can be shown by employing Smirnov's superposition principle for normal 1-currents (see \cite{Smirnov}, and even \cite{PS12}, \cite{PS13}, \cite{ARV} for variants in the setting of metric currents), as shown in \cite{Bia_Bo_2020} and \cite{PhdBonicatto} for $\bb$ with compact support. 

\begin{remark} [Stability of Lagrangian representations]\label{invaria_rep} {\rm It is not difficult to show that measures $\eeta$ inducing Lagrangian representations are stable
under restriction, namely whenever $B\in\Borel{\Upsilon}$, the measure $\tilde\eeta=\eeta\res B$ provides
a Lagrangian representation of a vector field $\tilde\rho\BB$, with $\tilde\rho\leq\rho$ defined by the right hand side
of \eqref{rhoBBrep} with $\tilde\eeta$ instead of $\eeta$. Accordingly, $D\cdot(\tilde\rho\BB)$
is representable by the right-hand side in \eqref{eq:concancellazioni}. On the other hand, the sum of Lagrangian representations $\eeta$ and $\eeta'$
of $\rho\BB$ and $\rho'\BB$ is not necessarily a Lagrangian representation of $(\rho+\rho')\BB$, as in \eqref{eq:nocancellazioni}
for $\mu=D\cdot ((\rho+\rho')\BB)$ a strict inequality might occur.
}\end{remark}

We now show, using \cite{ARV}, that the existence of Lagrangian representations still holds under the more general growth assumption \eqref{eq:growth} for densities $\rho\in\mathcal {R}^{\mathcal{I}}_+$ for some $\mathcal{I}$. We recall  that any $\rho\in\mathcal {R}^{\mathcal{I}}_+$ is extended to the full space by setting $\rho\mid_{t\in\R\setminus\mathcal{I}}\equiv0$. 

\begin{proposition}\label{extended_lagrangian}
Let $\rho\in\mathcal {R}^{\mathcal{I}}_+$ for some $\mathcal{I}$ bounded open interval, then any vector field $\rho\BB$ satisfying $D\cdot (\rho\BB)\in\Measures{\R^{d+1}}$ admits a Lagrangian representation.
\end{proposition}
\begin{proof}
In this proof we shall follow the language and notation of \cite{ARV}.
The 1-current  $S=\rho\BB\Leb{d+1}$ has locally finite mass, furthermore $\partial S\in\Measures{\R^{d+1}}$, which means that $S$ is a locally normal current. By \cite[Theorem 1.4]{ARV} it is possible to decompose the 1-current $S$ as the sum of a cyclic current $S_c$
and an acyclic current $S_a$ in such a way that 
    \begin{equation}\label{eq:decompositionmasscyclicacyclicm}
        \rho|\BB|\Leb{d+1}=\|S\|=\|S_a\|+\|S_c\|,
    \end{equation}
where $\|\cdot\|$ denotes the mass measure of a current. From \eqref{eq:decompositionmasscyclicacyclicm} we obtain that there exist 
$\rho_a , \rho_c \in L^{\infty}(\R^{d+1})$ nonnegative such that $\rho=\rho_a + \rho_c$, $S_a =\rho_a \BB$ and $S_c =\rho_c \BB$.
Next we show that $S$ is acyclic or, equivalently, that  $\rho_c=0$. 
    
Indeed, since $S_c$ is a cycle, one has $D\cdot(\rho_c\BB)=0$ in $\R^{d+1}$,
hence $\rho_c$ is a nonnegative solution of the continuity equation with initial datum 0.
In particular by \eqref{eq:weak1} we obtain 
$$\int_{\R^d}\phi\rho_{c,t}\,dx=\int_0^t\int_{\R^d}\bb_\tau\cdot\nabla\phi\rho_{c,\tau}\,dx\, d\tau
\quad\text{for all $\phi\in C^1_c(\R^d)$, $ t\geq0$.}$$
Fix $r>0$ and consider for $R\geq r$, $\phi_R\in C^\infty_c(\R^d)$ such that $0\leq\phi_R\leq 1$, $\phi_R\mid_{B_R(0)}= 1$ and 
$(1+|x|)|\nabla\phi_R(x)|\leq C$ for all $x\in\R^d$ for some fixed constant $C>0$ independent of $R$. These conditions grant, for all $t>0$,
\begin{eqnarray*}
\int_{B_r (0)}\rho_{c,t}\,dx&\leq&\int_{\R^d}\phi_R\,\rho_{c,t}\,dx\leq
\int_0^t\biggl|\int_{\R^d}\bb_s\cdot\nabla\phi_R\rho_{c,s}\,dx\biggr|\,ds\\
&\leq& C\|\rho\|_\infty\int_0^T \int_{\{\phi_R<1\}} \frac{|\bb^1_s|}{1+|x|}\,dx\,ds+
C\int_0^T \|\frac{|\bb^2_s|}{1+|x|}\|_\infty\int_{\{\phi_R<1\}} \rho_s\,dx\,ds,
\end{eqnarray*}
with the last term going to 0 as $R\to\infty$ by dominated convergence, implying that $\rho_c =0$.
We can then apply \cite[Theorem 1.4]{ARV} and find 
a positive Borel measure $\eeta$ over the space $\Gamma(\R^{d+1})$ of the open ended curves $\sigma(s)$ with 
locally finite length in $\R^{d+1}$ (namely their intersection with any bounded set of $\R^{d+1}$ has finite length)
such that $\eeta$-almost every curve $\sigma$ is injective and
\begin{align}
&S=\int_{\Gamma(\R^{d+1})} S_\sigma d\eeta(\sigma), \label{decomp} \\
&\|S\|=\int_{\Gamma(\R^{d+1})} \|S_\sigma\| d\eeta(\sigma),\label{masseq}\\
&\partial S=\int_{\Gamma(\R^{d+1})} \partial S_\sigma d\eeta(\sigma), \label{eq:ARVboundary}\\
& \|\partial S\|=\int_{\Gamma(\R^{d+1})} \|\partial S_\sigma\| d\eeta(\sigma).\label{eq:ARVvariationboundaryeq} 
\end{align}
where $S_\sigma$ denotes the current with locally finite mass in $\R^{d+1}$ canonically induced by $\sigma\in \Gamma(\R^{d+1})$.
Testing $S$ against $\varphi\frac{\BB}{|\BB|}$ for any $\varphi\in C_c(\R^{d+1})$ nonnegative, it is possible to show that $\eeta$-a.e. curve $\sigma$ has the property that $\sigma'(s)$ is a nonnegative multiple of $\rho\BB(\sigma(s))=\rho(\sigma(s))(1,\bb(\sigma(s))$ and $\rho(\sigma(s))>0$ for $\Leb{1}$-a.e. $s$ such that $|\sigma'(s)|>0$ (see for instance \cite{PhdBonicatto}).
Hence, up to a reparametrization, $\eeta$-almost every curve $\sigma$ can be represented as $(t,\gamma(t))$, with $\gamma$ a solution of \eqref{eq:ODE} in an interval $(t_\gamma^- ,t_\gamma^+)\subseteq \mathcal{I}$. 

From now on, with a slight abuse, we will see $\eeta$ as a measure
on the curves $\gamma$, noticing that $S_\sigma=S_{(t,\gamma)}$. The fact that $\eeta$ is concentrated on curves of locally finite length
in $\R^{d+1}$ implies that, for $\eeta$-a.e. curve $\gamma$, the possibly infinite limits $|\gamma(t_\gamma^+)| := \lim_{t \uparrow t_\gamma^+} |\gamma(t)| \quad \text{and} \quad |\gamma(t_\gamma^-)| := \lim_{t \downarrow t_\gamma^-} |\gamma(t)|$
exist.

Finally, we check that $\eeta$ is indeed concentrated on curves of finite length.
Testing $S$ against $\frac{1}{|\bb|(1+|x|)}(0,\bb)$, with $\tau_\gamma=(t_\gamma^++t_\gamma^-)/2$ we find 
\begin{align}
&\int\left|\log\left(\frac{1+|\gamma(t_\gamma^{\pm})|}{1+|\gamma(\tau_\gamma)|}\right)\right|\,d\eeta(\gamma)\leq\int\int_{t^-_{\gamma}}^{t^+_{\gamma}} \frac{|\dot\gamma(s)|}{1+|\gamma(s)|}ds\,d\eeta(\gamma)=\int\int_{t^-_{\gamma}}^{t^+_{\gamma}} \frac{|\bb(\gamma(s))|}
{(1+|\gamma(s)|)}ds\,d\eeta(\gamma)\nonumber \\
&= S\left(\frac{1}{|\bb|(1+|x|)}(0,\bb)\right)\leq
\|\rho\|_\infty \left\|\frac{|\bb^1|}{1+|x|}\right\|_{L^1((0,T)\times\R^d)}+\int_0^T\left\|\frac{|\bb^2_s|}{1+|x|}\right\|_\infty\int\rho_s\,dx ds
<\infty.
\end{align}
It follows that the first integral is finite, hence $\eeta$ is concentrated on curves of finite length.
We conclude by noting that $\eeta$ is a Lagrangian representation.
Indeed since $\eeta$ is supported on curves of finite length, \eqref{rhoBBrep} follows from \eqref{decomp} whereas \eqref{eq:ARVboundary} and \eqref{eq:ARVvariationboundaryeq} written in terms of $\gamma$ yield \eqref{eq:concancellazioni} and \eqref{eq:nocancellazioni} respectively.  
\end{proof}

\subsection{Untanglement} We next recall the notion of untangled family of curves $\Delta\subseteq\Upsilon$. Let us define
$$
\mathring{W}:=W_1  \cup  W_2\subseteq\Upsilon\times\Upsilon,
$$
where 
$$
W_1 := \left\{
(\gamma, \overline{\gamma})\in\Upsilon\times\Upsilon:\
\mathrm{Graph\,}\,\gamma
\cap
\mathrm{Graph\,}\,\overline{\gamma}
=\emptyset
\right\}
$$
and
$$
W_2:=\left\{(\gamma,\overline\gamma)\in\Upsilon\times\Upsilon:\ \mathrm{Graph\,}\,\gamma\cap \mathrm{Graph\,}\,\overline{\gamma}
=\mathrm{Graph\,}\gamma\vert_{(\max(t^-_\gamma,t^-_{\overline\gamma}),\min(t^+_\gamma,t^+_{\overline\gamma}))} \right\}.
$$   

\begin{definition}[Untangled families of curves and Lagrangian representations]\label{defuntangled}
A set $\Delta\subseteq\Upsilon$ is said to be untangled if
 $\Delta\times\Delta\subseteq \mathring{W}$. 
Analogously, a Lagrangian representation $\eeta$ is said to be untangled if there exists an untangled 
and $\sigma$-compact set $\Delta\subseteq\Gamma_\bb$ 
such that $\eeta$ is concentrated on $\Delta$.
\end{definition}

\begin{figure}
\includegraphics[width=12 cm]{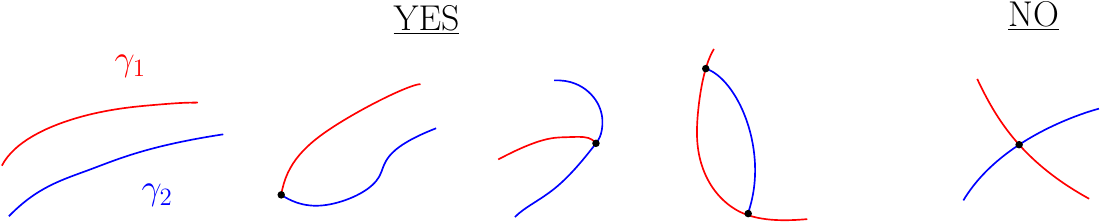}
\caption{Admissible and non-admissible curves in an untangled set}
\label{Disegni RLF}
\end{figure}

In other words, $\Delta$ is untangled if for any $\gamma,\,\overline{\gamma}\in\Delta$ either one graph is a continuation of the other which coincides on an open time interval,
or the two graphs do not intersect in their (relative) interior.
In \cite{Bia_Bo_2020}, under a $BV$ regularity assumption on $\BB$, the following fundamental description of Lagrangian representation
of $\rho\BB$ in terms of untangled families has been established. The theorem is stated in \cite{Bia_Bo_2020} assuming compactness
of the support of $\rho$, but the same proof works in our context, as thanks to Proposition~\ref{extended_lagrangian} our measures $\eeta$ are concentrated on curves with
finite length (under our standing assumptions \eqref{eq:growth} and $\rho\in\mathcal{R}^{\mathcal{I}}_+$ for some bounded open interval $\mathcal{I}$).
Indeed, the support of a Lagrangian representation $\eeta$ admits the decomposition as the union $\bigcup_{n\in\N} \Delta_n$,
where $(\Delta_n)_{n\in\N}$ is an increasing sequence of Borel subsets of $\Gamma_\bb$ such that each curve $\gamma\in\Delta_n$ is contained in $B_n(0)$. Then, Remark~\ref{invaria_rep} ensures that $\eeta\res\Delta_n$ is a Lagrangian representation of $\rho_n\BB$, with $\rho_n\leq \rho$
and \cite{Bia_Bo_2020} can be applied to all $\eeta\res\Delta_n$, allowing to obtain the following result for $\eeta$.

\begin{theorem}\label{teo_Bia_Bo}
    Let $\bb_t\in BV_{\rm loc}(\R^d,\R^d)$ with $|D\bb_t|(B_R)\in L^1(0,T)$ for all $R>0$ and assume that
    $D \cdot (\rho \BB)\in\Measures{\mathbb{R}^{d+1}}$.
    Then any Lagrangian representation $\eeta$ of $\rho\BB$ is untangled.
\end{theorem}

Note that on untangled families of curves $\Delta\subseteq\Upsilon$ the following equivalence relation is well-posed: 
\begin{equation}\label{equivalencerelation}
\begin{aligned}
\gamma \;\sim_\Delta\; \overline{\gamma}
\;\Longleftrightarrow\;
&\exists\, N \in \mathbb{N},\ \{\gamma_i\}_{i=1}^N \ \text{such that} \\
&\gamma = \gamma_1,\quad \overline{\gamma} = \gamma_N, \\
&\#\bigl( \mathrm{Graph}\,\gamma_i \cap \mathrm{Graph}\,\gamma_{i+1} \bigr) > 2,
\qquad i=1,\ldots, N-1 .
\end{aligned}
\end{equation}

Clearly, for any $a \in \Delta/\hskip -3pt \sim$ the union of the graphs of the curves in the equivalence class associated with $a$, 
forms the graph of a single curve $\sigma_a\in\Upsilon_0$, defined on an open interval $I_a$, which is a characteristic curve for $\bb$.
In the next proposition (see also \cite[Proposition 9.1 and 9.2]{Bia_Bo_2020}) we prove the Borel regularity of the natural maps
associated to this equivalence relation. We shall denote by $\mathscr{G}(\Delta)$ the union of the graphs of the curves in $\Delta$. Thanks to the
$\sigma$-compactness of $\Delta$ this set is easily seen to be Borel and of full $\rho\Leb{d+1}$-measure.

\begin{proposition}[Borel measurability of projections]\label{prop:meas}
Let $\Delta\subseteq\Upsilon$ be a $\sigma$-compact untangled family. Then 
the map $\gamma\mapsto\pi_\Delta(\gamma):=\sigma_{[\gamma]}\in\Upsilon_0$ is Borel. In addition, the map
\begin{equation}\label{eq:defP}
P_\Delta:  \mathscr{G}(\Delta)\mapsto\Upsilon_0
\end{equation}
which maps $(t,x)\in\mathrm{Graph\,}\gamma$ with $\gamma\in\Delta$ to $\sigma_{[\gamma]}$, is Borel.
\end{proposition}
\begin{proof}  Let us write $\Delta=\cup_p \Delta_p$ with $\Delta_p\subseteq \Delta_{p+1}$ compact, $p\geq 0$. Notice that
compactness of $\Delta_p$ grants a positive lower bound on the length of the intervals of curves $\gamma\in\Delta_p$. We denote by ${\bf J}_k$ the operator
associating to a pair $(\gamma,\tilde\gamma)\in\Delta\times\Delta$ with 
intersections of the domains having length larger than $2^{-k}$ the
new curve obtained by gluing the two curves. Since $\Delta$ is untangled, this construction is possible.

We now define recursively $\Delta_p^{(i)}$ as follows: $\Delta_p^{(0)}=\Delta_p$ and $\Delta_p^{(i+1)}$ is
obtained by applying the gluing ${\bf J}_i$ operator to admissible pairs of curves in $\Delta_p^{(i)}$ and $\Delta_p$.
It is not difficult to prove by induction on $i$ that all $\Delta_p^{(i)}$ are compact (here the lower bound on the length
of intersection plays a role), and that the graphs of all curves $\sigma_{[\gamma]}$
are finite or countable unions of graphs of curves in $\bigcup_{p,i}\Delta_p^{(i)}$.

Now, we fix $q$, $N$ with $N>q$ and we consider the map $\pi_{q,N}:\Delta_q\to\Upsilon$ which associates to a curve $\gamma\in \Delta_q$ the
curve whose graph is the union of the graphs of all curves in
$\bigcup_{p,i=0}^N \Delta_p^{(i)}$ containing $\gamma$ as subcurve. 
Thanks to the compactness of the domain and of the class of curves chosen to build $\pi_{q,N}(\gamma)$, this set-valued map is 
upper semicontinuous, hence Borel (see for instance \cite[Theorem II.20 and Corollary III.3]{Castaing}). It is also clear from
the construction that the graphs of $\pi_{q,N}(\gamma)$ are non decreasing with respect to $N$. 

Now, if we view $\pi_{q,N}$ as an $\Upsilon_0$-valued map, it is  Borel as well and it is clear that
$\pi_{q,N}(\gamma)$ converge as $N\to\infty$ to $\pi_\Delta(\gamma)=\sigma_{[\gamma]}$. This proves that the restriction
of the map $\gamma\mapsto \sigma_{[\gamma]}$ to $\Delta_q$ is Borel, and since $q$ is arbitrary we obtain that the map $\pi_\Delta$
is Borel on $\Delta$.

Finally, to prove that the map $P_\Delta$ is Borel we follow a similar argument, considering the $\Upsilon$-valued approximations
\begin{equation}
P_{q,N}:  \bigcup_{\gamma \in \Delta_q} \mathrm{Graph\,}\gamma\mapsto\Upsilon
\end{equation}

which maps $(t,x)\in\mathrm{Graph\,}\gamma$ with $\gamma\in\Delta_q$ to the curve $P_{q,N}(t,x)$ whose
graph is the union of all graphs in $\bigcup_{p,i=0}^N \Delta_p^{(i)}$ 
containing $(t,x)$.
To check that $P_{q,N}$ is Borel it is sufficient to show
(\cite[Theorem~III.2]{Castaing}) that  
$$
P_{q,N}^{-1}\bigl(\{\gamma: \gamma\cap U\neq\emptyset\}\bigr)
$$
is a Borel set for any open set $U\subseteq\R^{d+1}$. This set coincides with $$\mathscr{G}(\Delta_q)\cap\bigcup_{\sigma \in\bigcup_{p,i=0}^N \Delta_p^{(i)}\,\textit{s.t.}\,\mathrm{Graph\,}\sigma\cap U\neq\emptyset}\mathrm{Graph\,}\sigma$$ which is $\sigma$-compact, hence Borel.

\end{proof}

\subsection{Partition via characteristics}\label{sec_parvia}

Given an untangled Lagrangian representation $\eeta$, we consider an untangled and $\sigma$-compact set $\Delta$ where
 $\eeta$ is concentrated. Now, we can take as index set $\mathcal A$ the Polish space $\Upsilon_0$ and consider the canonical
 maps $\pi_\Delta:\Delta\to\mathcal A$ and $P_\Delta:\mathscr{G}(\Delta)\to\mathcal A$ of Proposition~\ref{prop:meas}. It will be technically
 convenient to endow $\mathcal A$ with the largest $\sigma$-algebra $\mathcal F$ that makes $P_\Delta$ measurable.

\begin{figure}
\includegraphics[width=6cm]{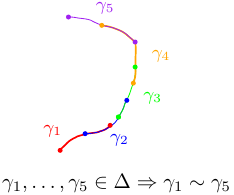}
{\bigskip} \label{Disegni RLF_2}
\caption{}
\end{figure}

It is then possible to disintegrate  $\eeta$ along the curves in $\Upsilon_0$ built by gluing the curves in the same equivalence class
and indexed by $a\in\mathcal A$ (we also use the notation $\sigma_a$ for the curve in $\Upsilon_0$ corresponding to $a$). Indeed, denoting ${\bf m}:=(\pi_\Delta)_\#\eeta\in\Measuresp{{\mathcal A}}$
and using Theorem~\ref{thm:disinte} we can first write
$\eeta=\int_{{\mathcal A}}\eeta_a\,d{\bf m}(a)$, with $\eeta_a\in\Probabilities{\Upsilon}$ concentrated for ${\bf m}$-a.e. $a\in\mathcal A$
in the curves $\gamma$ with
$\gamma\in\pi_\Delta^{-1}(a)$. Then, by considering the first component in \eqref{rhoBBrep}, we find
\begin{equation}\label{Eq1Lagrangianforuntangled}   
\rho\Leb{d+1}=\int_{\mathcal{A}}\left[\int_\Upsilon
(\mathrm{Id},\gamma)_\#(\Leb{1})\, d\eeta_a (\gamma)\right]\,d {\bf m}(a)
=\int_{\mathcal{A}}(\mathrm{Id},\sigma_a)_\#(w_a \Leb{1})\, d{\bf m}(a),
\end{equation}
with 
\begin{equation}\label{eq:formulawa}
w_a (t)=\int_\Upsilon\uno_{(t_\gamma^-,t_\gamma^+)}(t)\,d\eeta_a (\gamma)=\eeta_a (\{\gamma\in \pi_\Delta^{-1}(a):\, t\in (t_\gamma^-,t_\gamma^+) \})
\qquad t\in\R.
\end{equation}

Notice that, by construction, $w_a=0$ out of the domain $I_a$ of $\sigma_a$.
Moreover, it is not difficult to show (see \cite[Proposition 9.3]{Bia_Bo_2020}) that the properties of Lagrangian representations are transferred through this construction into the following additional identities:

\begin{equation}\label{Eq2Lagrangianforuntangle}
\mu=\int_{{\mathcal A}}\mu_a\,d{\bf m}(a)\qquad\text{and}\qquad\quad|\mu|=\int_{{\mathcal A}}|\mu_a|\,d{\bf m}(a).
\end{equation}
and 
\begin{equation}\label{eq_ODEwa}
    Dw_a=\mu_a\quad\textrm{for ${\bf m}$-a.e. $a\in\mathcal A$}
\end{equation}
where 
$$
\mu_a:=\int_\Upsilon\left(\delta_{(t_\gamma^-,\gamma(t_\gamma^-))}-\delta_{(t_\gamma^+,\gamma(t_\gamma^+))}\right)\, d\eeta_a (\gamma).
$$
In particular, \eqref{eq_ODEwa} implies  that $w_a \in BV(\R)$ for ${\bf m}$-a.e. $a\in\mathcal A$ (notice also that $w_a$ has
even pointwise bounded variation for ${\bf m}$-a.e. $a\in\mathcal A$, as the sum of the Dirac masses is finite thanks
to \eqref{eq:nocancellazioni}).

We will henceforth denote by $\tilde{w}_{a}$ the lower semicontinuous representative of $w_a$.

\smallskip

This way we have built a partition via characteristics, with $\mathcal A=\Upsilon_0$ and $P$ as in \eqref{eq:defP}, whose abstract definition is given below. In the abstract 
viewpoint, it is technically convenient to consider more general measurable spaces $(\mathcal A,\mathcal F)$, whose
properties (fully consistent with the previous construction) are listed below.

\begin{definition}[Partition via characteristics]
\label{def:partition_characteristics}
A partition via characteristics of $\rho\BB$ with $D \cdot (\rho\BB)\in\Measures{\R^{d+1}}$
consists of a Lagrangian representation $\eeta$ of $\rho\BB$ concentrated on a $\sigma$-compact untangled set $\Delta$, a Polish space $\mathcal A$ and a map $P:\mathscr{G}(\Delta)\to\mathcal A$ satisfying the following properties:
\begin{itemize}
    \item[(1)]the largest $\sigma$-algebra $\mathcal F$ of $\mathcal A$ that makes $P$ $(\Borel{\R^{d+1}},\mathcal F)$-measurable 
    contains $\Borel{\mathcal{A}}$,  the map $\pi_P:\Delta\to\mathcal A$
    \begin{equation}\label{eq:def_piP} 
    \pi_P(\gamma) := P(t,x)\quad\text{where $x=\gamma(t)$ for some $t\in (t_\gamma^-,t_\gamma^+)$, $\gamma\in\Delta$}
    \end{equation}
    is Borel and well defined on $\Delta$, i.e. invariant under the equivalence relation $\sim_\Delta$ in \eqref{equivalencerelation}, and setting ${\bf m}_P := (\pi_P)_\# \eeta$, $P^{-1}(a)$ is the graph of a characteristic curve $\sigma_a\in\Upsilon_0$
    in some open interval $I_a$ for ${\bf m}_P$-a.e. $a\in\mathcal A$;
    
    \item[(2)]  \eqref{Eq1Lagrangianforuntangled}, 
    \eqref{Eq2Lagrangianforuntangle}, \eqref{eq_ODEwa} hold with ${\bf m}={\bf m}_P$;
    \item[(3)] letting $(w_a)_{a\in\mathcal{A}}$ be the densities arising from the disintegration
    \begin{equation}\label{eqpartition1}
        \rho \mathscr{L}^{d+1} = \int_{\mathcal{A}} (\mathrm{Id}, \sigma_a)_\# (w_a \mathscr{L}^1) \, d{\bf m}_P(a)
    \end{equation}
    one has 
    \begin{equation}\label{eqpartition2}
        D w_a = \mu_a \in \Measures{\mathbb{R}}, \quad \text{for } {\bf m}_P\text{-a.e. } a \in \mathcal A,
    \end{equation}
    where $w_a$ is considered extended to $0$ outside $I_a$;
    \item[(4)] with $\mu=D \cdot (\rho\BB)$ it holds
    \begin{equation}\label{eqpartition3}
        \mu = \int_{\mathcal{A}} (\mathrm{Id}, \sigma_a)_\# \mu_a \, d{\bf m}_P(a) \quad \text{and} \quad 
        |\mu| = \int_{\mathcal{A}} (\mathrm{Id}, \sigma_a)_\# |\mu_a| \, d{\bf m}_P(a).
    \end{equation}
\end{itemize}
\end{definition}

\begin{definition}[Minimal partition by characteristics]
We say that a partition by characteristics is minimal if $\tilde{w}_a>0$ in $I_a$ for ${\bf m}_P$-a.e. $a\in
\mathcal A$.
\end{definition}

Since the condition $\tilde{w}_a>0$ is open, possibly splitting $\sigma_a$ in countably many subcurves where
$\tilde{w}_a>0$ in the interior of their domain $I_a$,
it is always possible to refine a partition by characteristics of $\rho\BB$ to produce a {\it minimal} partition via characteristics of $\rho \BB$
(while the explicit construction made before provides the {\it maximal} partition, given $\Delta$). 
For the sake of brevity we just sketch a construction, omitting measurability details. We consider in the product $\Upsilon_0\times\mathcal A$ the
set $G$ of pairs $(\sigma,a)$ with $I_\sigma$ connected component of $\{\tilde{w}_a>0\}$. Since for ${\bf m}_P$-a.e. $a\in\mathcal A$ the
set $\{\sigma:\ (\sigma,a)\in G\}$ is at most countable, we can cover $G$ (\cite[Theorem~18.10]{Kechris}) 
by a sequence of ${\bf m}_P$-measurable functions $f_n(a)$.
Then, choosing as index set $\mathcal A'=\mathcal A\times\N$, we may define 
$$
P_{\rm min}(t,x):=\left\{(a,n):\ \text{$n$ minimal integer such that $(f_n(a),a)\in G$, $(t,x)\in\mathrm{Graph}(f_n(a))$}\right\}.
$$

\begin{remark}\label{eta+eta'} {\rm Observe that if $\eeta,\,\eeta'$ are Lagrangian representations of $\rho\BB,\,\rho'\BB$ respectively that are concentrated on the
 same untangled set $\Delta$, then it is still possible to carry out the same construction presented above on $\eeta+\eeta'$ and obtain the equivalent of \eqref{eqpartition1} with $w_a \in BV(I_a)$, as it is the sum of $BV$ functions. Similarly, \eqref{eqpartition2} and the first equality in \eqref{eqpartition3} hold
 for $\mu=D\cdot((\rho+\rho')\BB)$.  On the other hand, the full validity of \eqref{eqpartition3} is not granted, as $\eeta+\eeta'$ might not be a Lagrangian representation for $(\rho+\rho')\BB$.}
\end{remark}

 \begin{lemma}\label{lemmasubsupports}
Let $\eeta$ be a Lagrangian representation concentrated on $\Delta\subseteq\Upsilon$, and let $\Delta'\supseteq \Delta$ be a $\sigma$-compact untangled set.
Let $P:\mathscr{G}(\Delta)\to\mathcal A$ and $P':\mathscr{G}(\Delta')\to\mathcal A'$ be the partitions via characteristics induced by $\Delta$ and $\Delta'$, respectively.
Then there exists a measurable map $U:\mathcal A\to\mathcal A'$, defined implicitly ${\bf m}_P$-almost everywhere by the identity
\[
\pi_{P'} = U \circ \pi_P,
\]
such that the preimage $U^{-1}(a')$ is countable for every $a'\in\mathcal A'$, and
\[
U_\# {\bf m}_P = {\bf m}_{P'}.
\]
Furthermore, considering the disintegration ${\bf m}_P={\bf m}_{P'}\otimes p_{a'}$ with $p_{a'}\in \Probabilities{\mathcal A}$ concentrated on $U^{-1}(a')$, 
one has
    \begin{equation}\label{formuladensitasottoclasse}
        \tilde{w}_{a'} (t)=\sum_{a\in U^{-1}(a')}\tilde{w}_{a}(t)p_{a'}(\{a\}) \qquad\text{for ${\bf m}_{P'}$-almost every $a'\in \mathcal{A}'$,}
    \end{equation}
    where $I_a\cap I_b=\emptyset$ for $a,\,b\in U^{-1}(a')$ with $a\neq b$, so that
    \begin{equation}\label{disjointdomains}
        \{t\in I_{a'}\,\mid \tilde{w}_{a'}(t)>0\}=\bigcup_{a\in U^{-1}(a')} \{t\in I_a\,\mid \tilde{w}_{a}>0\}.
    \end{equation}
    In particular,  for ${\bf m}_{P'}$-almost every $a'$, if $\tilde{w}_{a'}\vert_{I_{a'}}>0$, then  $U^{-1}(a')$ is a singleton.
\end{lemma}

\begin{proof}
      Since $\Delta' \supseteq \Delta$, curves in $\Delta$ which are not equivalent with respect to $\pi_\Delta$ might be equivalent with
      respect to $\pi_{\Delta'}$, because of the addition of curves. Hence, to ${\bf m}_{P}$-almost every  $a\in \mathcal{A}$ we may 
      associate $U(a):=a'$ such that $P^{-1}(a) \subseteq (P')^{-1}(a')$.
   
    Noting that for $a,\,b\in U^{-1}(a')$ with $a\neq b$ one has that $P^{-1}(a)$ and $P^{-1}(b)$ are disjoint subcurves of $(P')^{-1}(a')$, it follows that $U^{-1}(a')$ is at most countable. The measurability of $U$ is a consequence of the following fact: for any $\Omega\subseteq \mathcal{A}'$ one has $U^{-1}(\Omega)=P(P'^{-1}(\Omega)\cap \mathcal{G}(\Delta))$
and $(P')^{-1}(\Omega)\cap \mathcal{G}(\Delta)=\bigcup_{a\in U^{-1}(\Omega)}P^{-1}(a)$.  Hence, if $\Omega\in\mathcal F_{\mathcal{A}'}$, 
$P'^{-1}(\Omega)\cap \mathcal{G}(\Delta)$ is a Borel and $P$-saturated set, so that $P(P'^{-1}(\Omega)\cap \mathcal{G}(\Delta))$ is, by definition, an element 
of the $\sigma$-algebra $\mathcal F_{\mathcal{A}}$. We have 
$$U_\# {\bf m}_P=U_\#(\pi_P)_\# \eeta=(\pi_{P'})_\# \eeta={\bf m}_{P'}.$$
By disintegration we have ${\bf m}_{P}={\bf m}_{P'}\otimes p_{a'}$ with $p_{a'}\in\mathcal{P}(\mathcal{A})$ concentrated on
$U^{-1}(a')$ and
$$\rho \mathscr{L}^{d+1} = \int (\mathrm{Id}, \sigma_{a'})_\# \biggl(\sum_{a\in U^{-1}(a')} p_{a'}(\{a\}) w_a  \mathscr{L}^1\biggr) \, d{\bf m}_{P'}(a'),$$
so that the uniqueness of the disintegration of $\rho \mathscr{L}^{d+1}$ with respect to ${\bf m}_{P'}$ gives
     $$w_{a'}(t) =\sum_{a\in U^{-1}(a')} p_{a'}(\{a\})w_{a}(t) \qquad\text{for $\Leb{1}$-a.e. $t\in\mathbb{R}$} $$
    for ${\bf m}_{P'}$-a.e. $a'\in\mathcal{A}'$.
    Formula \eqref{formuladensitasottoclasse} is trivial at points $t\in I_a$ with $U(a)=a'$. At points
    in the intersection $\overline{I}_a\cap\overline{I}_b$ with $a\neq b$ and $U(a)=U(b)=a'$ it
    follows by noting that \eqref{eqpartition2} and \eqref{eqpartition3} imply that
\begin{equation*}
     |\mu| = \int_{\mathcal{\mathcal{A'}}} (\mathrm{Id}, \sigma_{a'})_\# |\mu_{a'}| \, d{\bf m}_{P'}(a')\leq 
     \int_{\mathcal{\mathcal{A'}}} (\mathrm{Id}, \sigma_{a'})_\# \sum_{a\in U^{-1}(a')} p_{a'}(\{a\})|\mu_{a}| \, d{\bf m}_{P'}(a')=|\mu|,
\end{equation*}
hence 
  \begin{equation*}
         |D w_{a'}|=|\mu_{a'}|=\sum_{a\in U^{-1}(a')} p_{a'}(\{a\})|\mu_{a}|  =\sum_{a\in U^{-1}(a')} p_{a'}(\{a\})|D w_{a}|\qquad\text{ for ${\bf m}_{P'}$-a.e. $a'\in\mathcal A'$,}
    \end{equation*}
    which implies
 \begin{equation}\label{eq_goodsumderivatives}
         (D w_{a'})^+ =\sum_{a\in U^{-1}(a')} p_{a'}(\{a\})(D w_{a})^+ \qquad\text{ for ${\bf m}_{P'}$-a.e. $a'\in\mathcal A'$.}
    \end{equation}
 
 Finally, if $\tilde{w}_{a'}\vert_{I_{a'}}>0$ and both \eqref{eq_goodsumderivatives} and \eqref{formuladensitasottoclasse} hold, then $U^{-1}(a')$ is a singleton. Indeed, the addends on the right-hand side of \eqref{formuladensitasottoclasse} have supports with disjoint interiors and \eqref{eq_goodsumderivatives} rules out cancellations of mass of the derivatives at the boundary.
\end{proof}

\begin{lemma}\label{lem_partizionieinsiemidicurve}
 Let $\eeta$ be a Lagrangian representation concentrated on an untangled set of curves $\Delta'$ and let $P'$ be the induced minimal partition via characteristics. Then, for any $\eeta$-negligible and $\mathcal{G}_\delta$ set $N$, $\Delta'\setminus N$ still produces the same partition, 
 up to a ${\bf m}_{P'}$-negligible set.
\end{lemma}
\begin{proof} We apply Lemma~\ref{lemmasubsupports} with $\Delta=\Delta'\setminus N$, a $\sigma$-compact untangled set where $\eeta$
is still concentrated. The minimality of the partition $P'$ then yields that $U^{-1}(a')$ is a singleton for ${\bf m}_{P'}$-a.e. $a'\in\mathcal A'$. 
\end{proof}

\begin{proposition}[Uniqueness of the minimal partition via characteristics]
Assume that $\rho\BB$ with $D\cdot(\rho\BB)\in\Measures{\mathbb{R}^{d+1}}$ admits only untangled Lagrangian representations. 
Then the minimal partition via characteristics $P$ of $\rho\BB$ is unique up to an isomorphism of the measure space
$(\mathcal{A},{\bf m}_P)$.    
\end{proposition}
\begin{proof} Consider two minimal partitions via characteristics $P_1$, $P_2$  and associated Lagrangian representations  $\eeta_1$, $\eeta_2$,
 respectively supported on $\sigma$-compact untangled sets of curves $\Delta_1$, $\Delta_2$. Since in Definition~\ref{defLagRep} the 
 dependence on the measure $\eeta$ is affine, convex combinations of $\eeta_1$ and $\eeta_2$ are still Lagrangian representations for 
 $\rho\BB$. Hence, with $\eeta=(\eeta_1+\eeta_2)/2$ we may find a $\sigma$-compact untangled set $\Delta_{12}$ where $\eeta$
 is concentrated. Hence, considering a $\mathcal{G}_\delta$ $\eeta$-negligible set $N\supseteq\Delta_1\cup\Delta_2\setminus\Delta_{12}$ 
 we obtain that
 $$\left(\Delta_1 \cup \Delta_2 \right)\setminus N \,\,\,\textrm{is untangled and induces a minimal partition via characteristics for $\rho\BB$.}$$
 Note that, by Lemma~\ref{lem_partizionieinsiemidicurve}, subtracting $N$ will not change the partitions induced by the sets $\Delta_i$, so we may assume, possibly replacing $\Delta_i$ by $\Delta_i\setminus N$, that $\left(\Delta_1 \cup \Delta_2 \right)$ is untangled and that the equivalence relation 
$\pi_{\Delta_1\cup\Delta_2}$ produces a minimal partition via characteristics $P'$.
 Let $U:\mathcal{A}\to \mathcal{A}'$ be as in Lemma~\ref{lemmasubsupports} 
 (with $\Delta=\Delta_1$, $\Delta'=\Delta_1\cup\Delta_2$) be such that $\pi_{P'}=U\circ \pi_{P_1}$.
 Considering as in \eqref{eqpartition1} the disintegrations
 $$\rho \mathscr{L}^{d+1} = \int_{\mathcal{A'}} (\mathrm{Id}, \sigma_{a'})_\# (w_{a'} \mathscr{L}^1) \, d{\bf m}_{P'}(a'),\quad
 \rho \mathscr{L}^{d+1} = \int_{\mathcal{A}} (\mathrm{Id}, \sigma_a)_\# (w_a \mathscr{L}^1) \, d{\bf m}_{P_1}(a)$$
by Lemma~\ref{lemmasubsupports} we have that, since $\tilde{w}_{a'}\vert_{I_{a'}}>0$, that 
$U^{-1}(a')=\{a\}$ and $I_a =I_{a'}$ for ${\bf m}_{P'}$-almost every $a'\in \mathcal{A}'$. 
Hence $U$ is a measurable isomorphism from $(\mathcal{A}',{\bf m}_{P'})$ to $(\mathcal{A},{\bf m}_{P_1})$.
 The conclusion then follows by the analogous argument applied to the decomposition induced by $\pi_{\Delta_2}$,
 taking the composition of the measurable isomorphisms.
\end{proof}

Since we have established a suitable uniqueness property for minimal partitions via characteristics, 
we will henceforth denote such partitions by $P_{min}$.  From now on, we shall also say that all Lagrangian representations related to $\bb$ are untangled if, for any bounded open interval $\mathcal{I}$ and any $\rho\in\mathcal{R}^{\mathcal{I}}_+$ such that $D\cdot (\rho\BB)\in\Measures{\R^{d+1}}$, every Lagrangian representation of $\rho\BB$ is untangled. Notice that, by Theorem \ref{teo_Bia_Bo}, this assumption holds for the class of $BV$ vector fields  considered.

\begin{proposition}\label{Untangledsupportofunion}
Let $\bb$ be such that all Lagrangian representations related to $\bb$ are untangled.
Let $\rho,\rho' \in \mathcal{R}^{\mathcal{I}}_+$ for some bounded open interval $\mathcal{I}$ with $D \cdot (\rho\BB) = \mu \in \Measures{\mathbb{R}^{d+1}}$ and $D \cdot (\rho'\BB) = \mu' \in \Measures{\mathbb{R}^{d+1}}$ and let $\eeta,\,\eeta'$ be Lagrangian representations for $\rho\BB,\,\rho'\BB$ concentrated on Borel sets $\Delta,\,\Delta'\subseteq \Upsilon$ respectively. 
Then $\Delta\cup\Delta'$ is untangled up to a set of negligible $(\eeta+\eeta')$-measure. 
\end{proposition}
\begin{proof} 

Fix a countable dense family $\{t_m\}_{m \in \mathbb{N}} \subset \mathcal{I}$
and consider the associated family of intervals $\{\mathcal{I}_n\}_{n \in \mathbb{N}}$ having these points as endpoints, we write
$\mathcal{I}_n = (t_n^-, t_n^+)$.

Let $T_n := \{ (t_\gamma^-, t_\gamma^+, \gamma) \in \Upsilon : \mathcal{I}_n \subseteq (t_\gamma^-, t_\gamma^+) \}$ and let $R_n : T_n \to \Upsilon$ be defined as $R_n(t_\gamma^-, t_\gamma^+, \gamma) = (t_n^-, t_n^+, \gamma)$.
Set $\eeta_n:=R_{n\#}(\eeta\res T_n)$ and  $\eeta'_n:=R_{n\#}(\eeta'\res T_n)$ respectively concentrated on $R_n(\Delta\cap T_n)$ and $R_n(\Delta'\cap T_n)$. A straightforward verification shows that the densities $\rho_n$ and $\rho'_n$, defined via the identities $\rho_n (t,\cdot) \Leb{d} = (e_t)_{\#} \eeta_n$ and $\rho'_n (t,\cdot) \Leb{d} = (e_t)_{\#} \eeta'_n$, satisfy the property that $\mu_n:=D\cdot (\rho_n \BB)$ and $\mu_n':=D\cdot (\rho'_n \BB)$ belong to
$\Measures{\R^{d+1}}$. Furthermore, the inclusions 
$$
\text{spt}(\mu_n^-) \cup \text{spt}(\mu_n '^-) \subseteq \{t = t_n^+\}\quad\text{and}\quad
\text{spt}(\mu_n^+) \cup \text{spt}(\mu_n'^+) \subseteq \{t = t_n^-\}
 $$ 
 imply that no cancellations occur at the level of the divergence. Consequently, $\eeta_n + \eeta'_n$ constitutes a Lagrangian representation for $\rho_n + \rho'_n$, the assumption on $\bb$ then implies that it has to be untangled. Hence, there exist $\sigma$-compact untangled
 sets $\Delta_n\subseteq\Upsilon$ where $\eeta_n+\eeta_n'$ is concentrated and with $N_n:=\Upsilon\setminus \Delta_n$ one has
 $(\eeta + \eeta')(R_n^{-1}(N_n)) =(\eeta_n+\eeta_n')(N_n)= 0$.  It follows that $\eeta+\eeta'$ is concentrated on the set $\Delta\cup \Delta'\setminus\bigcup_n R_n ^{-1} (N_n)$, and since $\bigcup_n T_n=\Upsilon$ one has
 $$
 (\Delta\cup\Delta')\setminus\bigcup_n R_n ^{-1} (N_n)=(\Delta\cup\Delta')\cap\bigcap_n \left(R_n^{-1}(\Delta_n)\cup T_n^c\right).
 $$
 Since each $\Delta_n$ is untangled, it follows that the set 
$\Delta \cup \Delta' \setminus \bigcup_n R_n^{-1}(N_n)$ is also untangled. 
Indeed, suppose by contradiction that there exist two elements 
$(t_\gamma^-, t_\gamma^+, \gamma)$ and $(t_{\gamma'}'^-, t_{\gamma'}'^+, \gamma') \in 
(\Delta \cup \Delta') \cap \bigcap_n R_n^{-1}(\Delta_n)$ 
which split at a time $\bar{t} \in (t_\gamma^-, t_\gamma^+) \cap (t_{\gamma'}'^-, t_{\gamma'}'^+)$. 
Then, there exists an $n \in \mathbb{N}$ such that 
$\bar{t} \in \mathcal{I}_n \subseteq (t_\gamma^-, t_\gamma^+) \cap (t_{\gamma'}'^-, t_{\gamma'}'^+)$. 
Consequently, we would have $R_n(t_\gamma^-, t_\gamma^+, \gamma)$,
$R_n(t_{\gamma'}'^-, t_{\gamma'}'^+, \gamma') \in \Delta_n$. 
This contradicts the fact that $\Delta_n$ is untangled, as these two elements split at $\bar{t}$.
\end{proof}

\begin{theorem}\label{BBL1BVexanduniqueness}
Let $\bb$ be such that all Lagrangian representations related to $\bb$ are untangled and let
$\rho \in {\mathcal R}^{\mathcal{I}}_+$ for some bounded open interval $\mathcal{I}$ with $D \cdot (\rho\BB) = \mu \in \Measures{\mathbb{R}^{d+1}}$. 
Let $P_{min}$ be the minimal partition via characteristics of $\rho\BB$ and let $\rho'$ be nonnegative with 
$\rho'\leq C\rho$ and $D \cdot (\rho'\BB) = \mu' \in \Measures{\mathbb{R}^{d+1}}$.
Then $P_{min}$ is a partition via characteristics for $\rho'\BB$. In particular, the curves of the minimal partition for $\rho'\BB$ can be assumed to be subcurves of those of $P_{min}$.
\end{theorem}

\begin{proof}  Note that, by setting $\rho+\rho' =u\rho$, we have $1\leq u\leq C+1$ on $\{\rho> 0\}$.
 Let $\eeta$, $\eeta'$ be untangled Lagrangian representations for $\rho\BB$ and $\rho'\BB$ concentrated on the sets $\Delta$ and $\Delta'$ and
 the induced minimal partitions $P_{min}$, $P'_{min}$, respectively.
 We may then write
$$\rho \mathscr{L}^{d+1} = \int_{\mathcal{A}} (\mathrm{Id},\sigma_a)_\# (w_a \mathscr{L}^1) \, d{\bf m}_{P_{min}}(a),\quad
\rho' \mathscr{L}^{d+1} = \int_{\mathcal{A}'} (\mathrm{Id},\sigma_{a'})_\# (w_a'\mathscr{L}^1) \, d{\bf m}_{P'_{min}}(a').$$
By Proposition~\ref{Untangledsupportofunion}, the set $\Delta \cup \Delta'$ is untangled up to a set of negligible $(\eeta + \eeta')$-measure,
hence, by Lemma~\ref{lem_partizionieinsiemidicurve}, applied separately on $\eeta$ and $\eeta'$, 
we may assume, without changing notation, that it is an untangled set. Starting from the identity 
$$u\rho\BB\Leb{d+1}=\int_\Upsilon(\mathrm{Id}, \gamma)_\#\left((1,\gamma') \Leb{1}\right)\,d(\eeta+\eeta')(\gamma) $$
taking into account Remark~\ref{eta+eta'} it is then possible to carry out the construction in Section~\ref{sec_parvia},
 producing a minimal partition 
$P''_{min}:\mathscr{G}(\Delta\cup\Delta')\to\mathcal{A}$ induced by $\Delta \cup \Delta'$ such that 
$$u\rho \Leb{d+1} = \int_{\mathcal A''} (\mathrm{Id}, \sigma_c'')_\# (w_c'' \Leb{1}) \, d{\bf m}''_{P''_{min}}(c),$$
with ${\bf m}''_{P''_{min}}=(\pi_{P''_{min}})_\#(\eeta+\eeta')$. Here the function $w''_c$, defined as in \eqref{eq:formulawa},
has finite pointwise variation for ${\bf m}''_{P''_{min}}$-almost every $c\in\mathcal A''$ and hence admits a lower semicontinuous representative
$\tilde{w}_c''$; the minimality of the partition grants that $\tilde{w}_c''>0$ only on an interval $I_c''$. Note that, in general, $P_{\min}''$ is not a partition via characteristics for $u\rho\BB$, although, being the untangled set larger, it is for both $\rho\BB$ and $\rho'\BB$.
Apply now Lemma~\ref{lemmasubsupports} to $\rho\BB$, with $P=P_{min}$ and $P'=P_{min}''$  to find
$U:\mathcal A \to \mathcal A''$ such that $\pi_{P''_{min}}=U\circ \pi_{P_{min}}$, 
and define the measure ${\bf m}_{P''_{min}}=~(\pi_{P''_{min}})_\# \eeta$.
Since  ${\bf m}_{P''_{min}} \leq {\bf m}''_{P''_{min}}$, we may write  ${\bf m}_{P''_{min}}=k \,{\bf m}''_{P''_{min}}$, with $k\in L^1(\mathcal{A}'', {\bf m}''_{P''_{min}})$ nonnegative. Applying Lemma~\ref{lemmasubsupports} we obtain
\begin{equation*}
\begin{split}
\rho \Leb{d+1} &= \int_{\mathcal{A}''} (\mathrm{Id}, \sigma_c'')_\# \sum_{a\in U^{-1}(c)} (p_c(a) w_a\Leb{1}) \, d{\bf m}_{P''_{min}}(c) \\
&=  \int_{\mathcal{A}''} (\mathrm{Id}, \sigma_c'')_\# \left(k(c)\sum_{a\in U^{-1}(c)} p_c(a) w_a  \Leb{1}\right) \,d{\bf m}''_{P''_{min}}(c).\end{split}
\end{equation*}
It follows that
$$
\int_{\mathcal{A}''} (\mathrm{Id}, \sigma_c'')_\# \left(k(c)\sum_{a\in U^{-1}(c)} p_c(a) w_a  \Leb{1}\right) \,d{\bf m}''_{P''_{min}}(c)\geq
\frac{1}{C+1}\int_{\mathcal A''} (\mathrm{Id}, \sigma_c'')_\# (w_c'' \Leb{1}) \, d{\bf m}''_{P''_{min}}(c),
$$
hence for ${\bf m}''_{P''_{min}}$-almost every $c\in \mathcal A''$ and every $t\in I''_c$
$$0<\frac{1}{C+1}\tilde{w}''_c(t)\leq k(c)\sum_{a\in P^{-1}(c)} p_c(a) \tilde{w}_a(t) $$
which implies by Lemma~\ref{lemmasubsupports} that there exists a single $a\in \mathcal A$ such that $\sigma_a=\sigma''_c$. That is, up to a bijection of $\mathcal{A}$, we have $P_{min}=P_{min}''$, which means that
$P_{min}$ is also a partition via characteristics for $\rho'$.\qedhere
 
\end{proof}

From Theorem~\ref{BBL1BVexanduniqueness} we can infer a direct proof of the well-posedness for the continuity equation.
Notice that the role of the ``dominating'' density $\bar\rho$ with $D\cdot(\bar\rho\BB)\in\Measures{\R^{d+1}}$ in the proof
is only instrumental for the determination of a good set of curves.

\begin{theorem}[Uniqueness for (CE)]\label{uniqueness_rho'}
Let $\bb$ be such that all Lagrangian representations related to $\bb$ are untangled and let
$\rho\in\mathcal{R}^{(0,T)}_+$  such that $D \cdot (\rho\BB) = \mu \in \Measures{\mathbb{R}^{d+1}}$ 
and
$$
{\rm spt} \, \mu \subseteq \{t=0\} \cup \{t=T\}
$$
then
\begin{equation}\label{eq1unicitasoluzioniCE}
        \mu^+ =\mu\res\{t=0\},\qquad\qquad \mu^- =-\mu\res\{t=T\}.
\end{equation}
Moreover,  in the strip $(0,T) \times \R^d$, the function $\rho$ is uniquely determined by 
the inner trace $\trplus{\rho\BB}{(1,0)}$ on $\Sigma=\{t=0\}=\partial\{t>0\}$. 
\end{theorem}

\begin{proof}
Let $\eeta$ be a Lagrangian representation of $\rho\BB$. From \eqref{eq:nocancellazioni} we infer
\begin{equation}\label{idTimes}
(0,T)= (t_\gamma^-, t_\gamma^+)
\qquad\text{for $\eeta$-a.e. $\gamma$,}
\end{equation}
since otherwise $\mu$ would have mass on intermediate times. Then \eqref{eq1unicitasoluzioniCE} follows immediately from 
\eqref{eq:concancellazioni} and \eqref{idTimes}. For any $\rho$ as above, consider $\tilde{\rho}(t,x):=\rho(t,x)+\uno_{ (-1,0)}(t) \rho_0(x)$, where $\mu^+ =\rho_0 \mathcal{H}^d \res _{\{t=0\}}$. Observe that 
\begin{equation}    
D\cdot(\tilde{\rho}\BB)=\rho_0\mathcal{H}^d \res _{\{t=-1\}}-\mu^-.
\end{equation}
Arguing as before, for any $\tilde\eeta$  Lagrangian representation of $\tilde\rho\BB$, it holds
\begin{equation}\label{Eq:dominioprolungamentocurve-1T}
(-1,T)= (t_\gamma^-, t_\gamma^+)
\qquad\text{for $\tilde\eeta$-a.e. $\gamma$,}
\end{equation}

Now, let $\rho'$ be any other solution to the continuity equation with the same initial condition, let 
$\bar\rho=\tilde\rho+\tilde\rho'$ and let $P_{min}$ be the unique minimal partition via characteristics of $\bar\rho\BB$ 
induced by an untangled family $\Delta$.
By Theorem~\ref{BBL1BVexanduniqueness}, $P_{min}$ is a partition via characteristics for $\rho\BB$. This implies that $\eeta$ is concentrated on curves obtained by restricting to the interval of time $(0,T)$ the characteristics  $(\sigma_a)_{a\in\mathcal{A}}$ associated to $P_{min}$, which by \eqref{Eq:dominioprolungamentocurve-1T} are defined on the interval $I_a=(-1,T)$. 
This means, using the notation of Definition~\ref{def:partition_characteristics} for $P_{min}$, that $\eeta$ is concentrated on the set
\begin{equation}\label{inclusionspteta'}
    Z:=\bigcup_{a\in\mathcal A}\left\{ \sigma_a\vert_{(0,T)} \right\}.
    \end{equation}

    In particular, $\eeta$ is concentrated on a disjoint family of curves in $[0,T)$  (here we invoke the untanglement of the set 
    $\Delta$, see Definition~\ref{defuntangled}).
We can now conclude by noting that we can disintegrate $\eeta$ with respect to the evaluation at the time $t=0$ as
$$\eeta=\int \delta_{\gamma_x}\,d(e_0)_\#\eeta(x)$$
where $\gamma_x$ is the unique curve in $Z$ (and there is one, for $(e_0)_\#\eeta$-a.e. $x$), such that $\gamma_x (0)=x$. 
Then, being $\rho_t \mathcal{L}^d=(e_t)_\# \eeta$ for all $t \in (0,T)$, we have that $\rho$ is uniquely determined by the measure 
$(e_0)_\# \eeta$, which in turn is given by
$$(e_0)_\#\eeta=\trplus{\rho\BB}{(1,0)}\mathscr{H}^d\res\Sigma,$$
as an easy consequence of \cite[Lemma 5.6]{Bia_Bo_2020} (see also Remark~\ref{rem_negative_times}).
 Since the argument is symmetric, reversing the roles of $\rho$ and $\rho'$ we obtain that $\rho=\rho'$.
\end{proof}

We now show that Theorem \ref{uniqueness_rho'} can also be used to establish uniqueness for a certain notion of solution to the transport equation.
More precisely, we introduce a suitable notion of one-sided nearly incompressible vector field, allowing the density $\rho $ to vanish.
In analogy with the nearly incompressible case, we then define the corresponding notion of $\rho$-weak solution to the transport equation and prove uniqueness for a given initial datum (prescribed in the distributional sense as in \eqref{eq:initial_cond_we}, see also Remark~\ref{rem_negative_times}).

\begin{definition}[One sided nearly incompressible fields $\bb$ and $\rho$-weak solutions]\label{def1sidednearlyincompressible}
A vector field $\bb: (0,T)\times\R^d\to \R^d$ satisfying \eqref{eq:growth} is called one-sided nearly incompressible if there exists $\rho\in\mathcal{R}^{(0,T)}_+$, $\rho\not\equiv0$, satisfying
$D\cdot (\rho\BB)=0$ in $(0,T)\times\R^d$. \\
 A function $u\in L^\infty(\R^{d+1})$ is said to be a $\rho$-weak solution of  the transport equation if
   $$D\cdot (u\rho\BB)=0\qquad\text{in $(0,T)\times\R^d$}.$$
\end{definition}
Clearly, uniqueness for $\rho$-weak solutions is possible (and will be intended) only up to modifications in $\rho \mathcal{L}^{d+1}$-negligible sets.
\begin{corollary}[Uniqueness]
 Let $\bb$ be such that all Lagrangian representations related to $\bb$ are untangled and let $\rho\in {\mathcal R}^{(0,T)}_+$ as in Definition~\ref{def1sidednearlyincompressible}.
 Then, for any $u_0 \in L^\infty(\R^d)$, any $\rho$-weak solution $u$  such that
 \begin{equation}\label{eq:initial_cond_we}
 D\cdot (u\rho\BB)=u_0 \rho_0\Haus{d}\res \{t=0\}\qquad\text{in $(-\infty,T)\times\R^d$}
 \end{equation}
 is uniquely determined $\rho\mathcal{L}^{d+1}$-a.e. . 

\end{corollary}
\begin{proof}
 The proof follows by applying Theorem~\ref{uniqueness_rho'} (to the nonnegative function $(u+M)\rho$, with $M \ge \|u\|_{L^\infty})$.
 \qedhere
\end{proof}

We recall however that the approach in the original work \cite{Bia_Bo_2020} was to establish these well-posedness results through renormalization along
characteristics (while in Section~\ref{sec:sbd} of this paper, dealing with vector fields of bounded deformation, renormalization is
computed at the Eulerian level). Indeed, let us see how Theorem~\ref{BBL1BVexanduniqueness} also allows to deduce a chain rule formula by basically factorizing the problem into 1-dimensional ones. \\
To begin with, notice that, given a partition via characteristics for $\rho\BB$, from \eqref{eqpartition1} we obtain
the distributional identities
$$
v\rho\BB=\int_{\mathcal{A}} (\mathrm{Id}, \sigma_a)_\# (v_aw_a (1,\sigma_a')\mathscr{L}^1) \, d{\bf m}(a),
\qquad
D\cdot (v\rho\BB)=\int_{\mathcal{A}} (\mathrm{Id}, \sigma_a)_\#\frac{d}{dt}(v_a w_a) \,d\mathbf{m}(a)
$$
for any $v \in L^\infty({\rm supp}(\rho))$, with $v_a:=v \circ\sigma_a$. 
If $\rho'=v\rho$ fits into the setting of Theorem~\ref{BBL1BVexanduniqueness}, we have some additional structure, namely the following.

\begin{lemma}\label{lem:uw}
Let $\bb,\,\rho, \, \rho'$ be as in Theorem \ref{BBL1BVexanduniqueness}, and let $u \in L^\infty(\R^{d+1})$ be such that $\rho'=u\rho$. Let
$\mathbf{m}$, $(\sigma_a)_{a \in \mathcal{A}}$, $(w_a)_{a \in \mathcal{A}}$ be a  partition via characteristics of $\rho\BB$. 
Then 
$$
u_a w_a \in BV(I_a)\qquad\text{for $\mathbf{m}$-a.e. $a \in \mathcal{A}$}.
$$
\end{lemma}

\begin{proof}
Thanks to Theorem~\ref{BBL1BVexanduniqueness}, $P$ is also a partition via characteristics for $\rho'\BB$ (as this is true for the minimal partition $P_{\rm min}$). Then its weight function $w_a' \in BV(I_a)$ along the curves of the partition  must agree with $u_a w_a$ for $\mathbf{m}$-a.e. $a \in \mathcal{A}$.\qedhere
\end{proof}

\begin{proposition}\label{prop:chainrule}
Let $h\in C^1(\R)$,  $\bb$ be such that all Lagrangian representations related to $\bb$ are untangled and $\rho\in\mathcal{R}^{\mathcal{I}}_+$ for some bounded open interval $\mathcal{I}$. Let ${\bf m}$, $(\sigma_a)_{a \in \mathcal{A}}$, $(w_a)_{a \in \mathcal{A}}$, $(\mu_a)_{a \in \mathcal{A}}$ be associated to the minimal partition via characteristics $P_{min}$ for $\rho\BB$. Let also $u \in L^\infty(\R^{d+1})$ be such that $u_a w_a \in BV_{\rm loc}(I_a)$ for $\mathbf{m}$-a.e. $a \in \mathcal{A}$, and denote $\mu_a':=D(u_a w_a)$. \\
 Then, for $\mathbf{m}$-a.e. $a \in \mathcal{A}$ the functions $u_a$ and $h(u_a)w_a$ belong to $BV_{\rm loc}(I_a)$ and
$D(h(u_a) w_a)$ is representable by
    \begin{align}
        &\sum_i \,[h(u_a(t_i^+))w_a(t_i^+)-h(u_a(t_i^-))w_a(t_i^-)]\delta_{t_i}+h'(\tilde{u}_a) (Du_a)^{\rm cont}\tilde{w}_a+h(\tilde{u}_a)(Dw_a)^{\rm cont}
        \label{eq:mu_a_h} \\
        &=\sum_i \, [h(u_a(t_i^+))w_a(t_i^+)-h(u_a(t_i^-))w_a(t_i^-)]\delta_{t_i}+h'(\tilde{u}_a)(\mu'_a)^{\rm cont}+(h(\tilde{u}_a)-\tilde{u}_ah'(\tilde{u}_a))\mu_a^{\rm cont},\nonumber
    \end{align}
where $\{t_i\}_i$ are jump points either of $u_a$ or of $w_a$.
\end{proposition}

\begin{proof}
As $P_{min}$ is a minimal partition for $\rho\BB$, we have $\inf_{J_a}\tilde{w}_a>0$ on any compact sub-interval $J_a$ of $I_a$, thus 
Vol'pert's Chain Rule and the Leibniz formula for the products of locally bounded $BV$ functions (see e.g. \cite[Theorem~3.96, Example~3.97]{AFP_00})
yield that $w_a^{-1} \in BV_{\rm loc}(I_a)$ and $u_a=w_a^{-1}(u_aw_a)\in BV_{\rm loc}(I_a)$ as well.
The formula \eqref{eq:mu_a_h} is then a simple consequence of the chain rule applied to $h(u_a)$ and of the Leibniz rule
applied to $h(u_a)$ and $w_a$.\qedhere
\end{proof}

\begin{corollary}[Chain Rule]\label{chain_BB}
Under the assumptions of Proposition~\ref{prop:chainrule} one has 
$$
D\cdot(h(u)\rho\BB\mathscr{L}^{d+1})=\int (\mathrm{Id},\sigma_a)_\# D(h(u_a)w_a) \, d\mathbf{m}(a),
$$
for any $h \in C^1(\R)$, with $D(h(u_a)w_a)$ representable as in \eqref{eq:mu_a_h}.
\end{corollary}

Now, putting together Lemma \ref{lem:uw} and Corollary \ref{chain_BB}, we obtain another proof of the uniqueness for (CE). Indeed,  pick $\rho_1$, $\rho_2 \in \mathcal{R}_+$ with the same datum at $t=0$ such that $D \cdot (\rho_i\BB)$ is a measure vanishing in $(0,T) \times \R^d$ ($i=1,2$). Then, we have that $u_a w_a \in BV(I_a)$ when $u=\frac{\rho_1-\rho_2}{\rho_1+\rho_2}$, by applying Lemma \ref{lem:uw} with $\rho=\rho_1+\rho_2$ and $\rho'=\rho_i$ for $i=1,2$. Then, exploiting the disintegration \eqref{eq:mu_a_h}, we can argue as in the proof of Theorem~\ref{thm:main2}, choosing $h$ as a $C^1$ approximation of $h(z)=z^+$.
Indeed, in the case $h(z)=z^+$, an application of Lemma \ref{lem:cylindrical} to $(\rho_1-\rho_2)^+ \in \mathcal{R}_+$ in combination with Corollary \ref{chain_BB} yields that $\int (\rho^1_t-\rho^2_t)^+\,dx=0$ for all $t \in (0,T)$ if we establish that (for $\mathbf{m}$-a.e. $a$)
$$\sum_i[h(u_a(t_i^+))w_a(t_i^+)-h(u_a(t_i^-))w_a(t_i^-)]\delta_{t_i}+h'(\tilde{u}_a)(\mu'_a)^{\rm cont}+(h(\tilde{u}_a)-\tilde{u}_ah'(\tilde{u}_a))\mu_a^{\rm cont}=0.$$
Here the second term vanishes since $\mu_a'=0$ and the third term vanishes by the choice of $h$. Finally, for the first term, we notice that applying Corollary \ref{chain_BB} to $\rho_j \in \mathcal{R}_+$ ($j=1,2$) the condition $D \cdot (\rho_j\BB)=0$ in $(0,T) \times \R^d$ gives that
$$\frac{\rho_j(\sigma_a(t_i^+))}{(\rho_1+\rho_2)(\sigma_a(t_i^+))}w_a(t_i^+)=\frac{\rho_j(\sigma_a(t_i^-))}{(\rho_1+\rho_2)(\sigma_a(t_i^-))}w_a(t_i^-)$$
for all $t_i$ discontinuity points either of $\frac{\rho_j}{\rho_1+\rho_2} \circ \sigma_a$ or $w_a$, which proves the continuity of $h(u_a)w_a$ in such points. \\
We highlight the fact that the non-negativity of $\rho_1$, $\rho_2$ is really a necessary condition for the validity of the result 
(recall Example~\ref{example:nogood}, as the vector field $\bb(x)={\rm sign}(x)$ of the example is one-sided nearly incompressible).

\bibliographystyle{siam}
\small
\bibliography{Bibliography}

\end{document}